\documentclass[11pt]{article}
\usepackage{ae,lmodern} 
\usepackage[latin1]{inputenc}
\usepackage[T1]{fontenc}
\usepackage{amsmath}
\usepackage{amsthm}
\usepackage{mathtools}
\usepackage{amscd}
\usepackage{amssymb}
\usepackage{amsfonts}
\usepackage{graphics}
\usepackage[dvipsnames]{xcolor}
\usepackage{cancel}
\usepackage{geometry}
\usepackage{cite}
\usepackage{enumitem}
\usepackage{soul} 
\usepackage{mathabx}
\usepackage[pagewise]{lineno}
\usepackage{cite}
\mathtoolsset{showonlyrefs}

\setenumerate[2]{label=(\alph*)}
\setenumerate[1]{label=\textup{(\roman*)}}
\setenumerate[3]{label=(\arabic*)}

\newcommand{\normW}[1]{{\vert\kern-0.25ex\vert\kern-0.25ex\vert #1 
		\vert\kern-0.25ex\vert\kern-0.25ex\vert}}
\newcommand{\normWbig}[1]{{\Big\vert\kern-0.25ex\Big\vert\kern-0.25ex\Big\vert #1 
		\Big\vert\kern-0.25ex\Big\vert\kern-0.25ex\Big\vert}}

\newcommand{\C}{\mathbb{C}}

\newcommand{\R}{\mathbb{R}}

\newcommand{\Z}{\mathbb{Z}}

\newcommand{\boC}{\mathcal{C}}

\newcommand{\boE}{\mathcal{E}}

\newcommand{\boX}{\mathcal{X}}

\newcommand{\NE}{\mathcal{NE}(\mathbb{R})}
\newcommand{\Emin}{E_\textup{min}}

\newcommand{\bu}{\mathbf{u}}

\renewcommand{\d}{\,\mathrm{d}}

\DeclareMathOperator{\supp}{{\rm supp}}

\newcommand{\loc}{\operatorname{loc}}

\theoremstyle{plain}
\newtheorem{theorem}{Theorem}[section]
\newtheorem{proposition}[theorem]{Proposition}
\newtheorem{lemma}[theorem]{Lemma}

\theoremstyle{definition}
\newtheorem{definition}[theorem]{Definition}

\newtheorem{remark}[theorem]{Remark}

\theoremstyle{remark}

\numberwithin{equation}{section}

\begin{document}
	\title{Momentum--mass normalized dark--bright solitons to one dimensional Gross--Pitaevskii systems}                                 
	\author{
		\renewcommand{\thefootnote}{\arabic{footnote}}
		Salvador  L\'opez--Mart\'inez\footnotemark[1]}
	\footnotetext[1]{
		Departamento de Matem\'aticas, Universidad Aut\'onoma de Madrid, Ciudad Universitaria de Cantoblanco, 28049, Madrid, Spain.\\
		E-mail: {\tt salvador.lopez@uam.es}}
	\date{}
	\maketitle

	\begin{abstract}
    	We rigorously establish the existence of dark--bright solitons as traveling wave solutions to a one dimensional defocusing Gross--Pitaevskii system, a widely used model for describing mixtures of Bose--Einstein condensates and nonlinear optical systems. These solitons are shown to exhibit symmetry and radial monotonicity in modulus, and to propagate at subsonic speed. Our method relies on minimizing an energy functional subject to two constraints: the mass of the bright component and a modified momentum of the dark component. The compactness of minimizing sequences is obtained via a concentration--compactness argument, which requires some novel estimates based on symmetric decreasing rearrangements.
	\end{abstract}   
	
	\maketitle
	
	\medskip
	\noindent{{\em Keywords:}
		Defocusing Schr\"odinger system, Gross--Pitaevskii system, dark--bright solitons, non standard conditions at infinity, concentration--compactness, symmetric decreasing rearrangements.
		
	
		\section{Introduction}\label{intro}

The aim of this paper is to analyze the existence of solitonic solutions to the following  system:
\begin{equation}\label{GPS}
\begin{cases}
i\partial_t\Psi = \partial_{xx}\Psi + (1-|\Psi|^2 - \alpha |\Phi|^2)\Psi, &(x,t)\in\mathbb{R}^2,
\\
i\partial_t\Phi = \partial_{xx} \Phi - (\alpha|\Psi|^2 + \beta |\Phi|^2)\Phi, &(x,t)\in\mathbb{R}^2,
\end{cases}
\end{equation}
where $\Psi$ and $\Phi$ are complex--valued functions, and $\alpha, \beta$ are real parameters. This system consists of two coupled nonlinear Schr\"odinger equations in one spatial dimension and is commonly referred to as the Gross--Pitaevskii system. This is a classical model for describing the dynamics of quasi--one dimensional Bose--Einstein condensates (BECs) consisting of a mixture of two different species \cite{kevrekidis_frantzeskakis_carretero-gonzalez_2008,berloff2008}, as well as for the multi--mode propagation of pulses in nonlinear optical fibers \cite{kivshar1998dark,KivsharAgrawal}. The sign of $\alpha$ determines the nature of the mutual interactions between species: for $\alpha > 0$, the interaction is repulsive (defocusing), while for $\alpha < 0$, it is attractive (focusing). Similarly, $\beta > 0$ corresponds to self--defocusing interactions in the second component, and $\beta<0$ to self--focusing ones.  The case $\beta=0$ is special, as the second equation becomes linear in $\Phi$. In this regime, the system is sometimes referred to as the Gross--Clark system, which models the motion of an uncharged impurity (represented by the wave function $\Phi$) within a Bose--Einstein condensate \cite{Grant_Roberts}.

In both BECs and optics, coherent structures {known} as \emph{solitons} have been experimentally observed. These are localized waves that maintain their shape while traveling at a constant velocity. Based on their qualitative properties, solitons are formally classified into different types. For instance, localized density bumps that vanish at infinity are called \emph{bright solitons}, while localized density dips with nonzero background are referred to as \emph{dark solitons}. 

Solitonic pairs consisting of combinations of dark and bright components (and other structures such as antidark solitons or domain walls) have also been observed in mixtures of BECs and nonlinear optical systems, though their specific configurations depend on the nature of the internal and mutual interactions \cite{KevNisFraMalCar}. For instance, it is well known that bright solitons do not exist in the scalar $(\alpha=0)$ defocusing $(\beta>0)$ case. However, if both internal and mutual interactions are repulsive, i.e. $\alpha>0$ and $\beta>0$, then the  existence of dark--bright soliton pairs is physically justified by a symbiotic relationship between the components, that is, the dark component effectively \emph{supports} the bright one. 

In this work, we focus on the mathematical analysis of dark--bright soliton pairs, i.e. solutions to \eqref{GPS} that satisfy the conditions at infinity:
\begin{equation}\label{condinfty}
|\Psi(x,\cdot)|\to 1,\quad |\Phi(x,\cdot)|\to 0,\quad\text{as }|x|\to\infty.
\end{equation}
Under these conditions, we can define the \emph{energy} functional, which is formally conserved along the flow:
\begin{equation}\label{eq:originalenergy}
\int_{\mathbb{R}}\Big( \frac{1}{2}|\partial_x\Psi|^2+\frac{1}{2}|\partial_x\Phi|^2+\frac{1}{4}(1-|\Psi|^2)^2+\frac{\beta}{4}|\Phi|^4 +\frac{\alpha}{2}|\Phi|^2|\Psi|^2 \Big)\d x,
\end{equation}
where time dependence is implicit. The energy functional will play a fundamental role in our analysis. In fact, we will work in the functional framework that keeps the energy finite. This topic will be developed in next section.

The motion of the solitons at constant velocity while maintaining shape is achieved mathematically by solutions to \eqref{GPS} of the form:
\[
\Psi(x,t) = u(x-ct),\quad \Phi(x,t) = e^{i((\lambda+c^2/4)t-cx/2)}v(x-ct),
\]
for some $c,\lambda\in\mathbb{R}$, where $u$ and $v$ are complex--valued functions which represent \emph{traveling waves} with velocity $c$. The equations for $u$ and $v$ are:
\begin{equation}\label{TWS}
\begin{cases}
icu' + u'' + (1-|u|^2-\alpha|v|^2)u=0, &x\in\mathbb{R},
\\
-v''=(\lambda -\alpha|u|^2-\beta|v|^2)v, &x\in\mathbb{R}.
\end{cases}
\end{equation}
We keep imposing the asymptotic conditions:
\begin{equation}\label{condinftyuv}
|u(x)|\to 1,\quad |v(x)|\to 0,\quad\text{as }|x|\to\infty.
\end{equation}
The system \eqref{TWS} with \eqref{condinftyuv} is invariant under translations in $x$ and the phase transformations $(u,v)\mapsto (e^{ik}u, e^{i\ell}v)$ for any $k,\ell\in\mathbb{R}$. Moreover, by complex conjugation, we may assume without loss of generality that $c \geq 0$. Also, since the second equation in \eqref{TWS} has
only real coefficients, we will restrict our attention to solutions whose second component $v$
is real--valued. In fact, it will be shown in the Appendix that the bright component to  system \eqref{TWS} is real--valued up to a multiplication by a complex constant of modulus one. Notice that this is no longer the case for other solitonic pairs such as dark--dark solitons, for which both components have nontrivial imaginary part, see \cite{Sheppard_Kivshar} for instance.

We now examine the particular case $\alpha=0$, in which the equations are decoupled. First of all, the second equation becomes
\begin{equation}\label{eq:v}
-v'' = (\lambda - \beta v^2)v,\quad x\in\mathbb{R}.
\end{equation}
The theory of bright solitons in the focusing case $\beta<0$ is a classical subject and the references are extensive, we only refer to \cite{cazenave}. In contrast,  the defocusing equation for $\beta\geq 0$ admits a unique solution in $H^1(\R)$, which is $v = 0$. 

On the other hand, the first equation
\begin{equation}\label{eq:u}
icu' + u'' + (1-|u|^2)u=0,\quad x\in\mathbb{R},
\end{equation}
corresponds to the traveling wave profile equation for the defocusing Gross--Pitaevskii equation. It is well--known that, for every subsonic speed, i.e. for every $c\in [0,\sqrt{2})$, there exists a unique (up to translations and phase shift) finite--energy dark soliton solution to \eqref{eq:u}, which is explicit, and its stability properties have been extensively studied \cite{linbubbles,DiMenzaGallo,orbitalblack,asymptoticgray,asymptstabblack}. We provide more details on this subject in next section. For a general overview, see the survey \cite{bethuel0}.

Of course, the pair $(u,0)$, where $u$ is a nontrivial solution to \eqref{eq:u}, trivially satisfies \eqref{TWS} even for $\alpha\not=0$. We refer to such solutions as \emph{semi-trivial}, while \emph{trivial solutions} correspond to $(e^{ik},0)$ for some constant $k\in\mathbb{R}$. 

Significant efforts, coming mainly from the physics community, have been devoted to obtaining dark--bright solitons as solutions to the system \eqref{GPS} in the coupled case $\alpha\not=0$. However, much of the existing work relies on numerical methods or formal expansions \cite{kivshar1998dark, kevrekidis_frantzeskakis_carretero-gonzalez_2008}. Despite its obvious relevance for theoretical and experimental physics, a rigorous derivation of fully nontrivial dark--bright solitons as solutions to the Gross--Pitaevskii system is missing in the literature. To the best of our knowledge, only partial results are available. For instance, the case $\beta=0$ has been dealt with in \cite{Maris2006,alhelou}. On the other hand, it is well known that the system \eqref{TWS} with the choice $\alpha = \beta = 1$, referred to as Manakov system \cite{Manakov}, is integrable, and explicit fully nontrivial solutions exist \cite{Busch_Anglin,Sheppard_Kivshar}. However, the physically relevant range of parameters extends beyond the integrable case. A broader, though still limited, range for $\alpha$ and $\beta$ is considered in the recent work \cite{Mao_Zhao}, where semi--explicit dark--bright solitons are constructed via a Lagrangian variational method.

The main contribution of this work is an analytical proof of the existence of dark--bright soliton solutions to \eqref{GPS} in the {so--called \emph{miscibility regime} i.e. for any value of the parameters $\alpha>0$ and $\beta>0$ satisfying $\alpha^2\leq\beta$. Moreover, we show that these solutions are symmetric and radially nonincreasing in modulus, and they propagate at subsonic speed. In the context of mixtures of BECs, the condition $\alpha^2\leq\beta$ indicates that the repulsive forces between the two species in the mixture of condensates are weaker than the repulsive interatomic interactions \cite{kevrekidis_frantzeskakis_carretero-gonzalez_2008,berloff2008}. Mathematically, in the miscibility regime, the renormalized energy functional \eqref{eq:energy} defined below is nonnegative, as proved in Theorem~\ref{thm:properties}. This property plays a crucial role in establishing the existence result. In contrast, in the \emph{immiscibility regime} corresponding to $\alpha^2 > \beta$, the renormalized energy is expected to attain negative values, and the problem therefore requires a different analytical approach. This case will be addressed in  future work.

Our approach relies on minimizing the mentioned} energy functional under two constraints, one for the momentum of the dark component, and another one for the mass of the bright component, in the spirit of \cite{alhelou}. Arguing this way, the constants $c$ and $\lambda$ arise as Lagrange multipliers, one corresponding to each constraint. 

The proof of the compactness of minimizing sequences is based on the well--known Concentration--Compactness principle of Lions \cite{CCPLionsI,CCPLionsII}. This approach requires excluding the potential issues of \emph{vanishing} and \emph{dichotomy}. Ruling out dichotomy for functionals of more than one variable is not a trivial task, and the literature in this matter is limited. To address this challenge, the use of symmetric decreasing rearrangements (see Section~\ref{sec:rearrangements}) has been proven effective. Indeed, a refined P\'olya--Szeg\H{o} inequality due to \cite{Garrisi} was the key for establishing compactness for {systems} of Klein--Gordon equations. This technique has also been successfully employed in various other contexts, always for solutions to systems whose both components are real--valued and vanish at infinity. These include nonlinear Schr\"odinger--Korteweg--de Vries systems \cite{Albert_Bhattarai}, focusing nonlinear Schr\"odinger systems \cite{Nguyen_Wang}, and higher--dimensional problems \cite{Shibata,GouJeanjean2016,GouJeanjean2018}. 

In the case of dark--bright solitons, the use of symmetric decreasing rearrangements presents a particular challenge, as the dark component does not vanish at infinity. To address this, nontrivial manipulations of the energy and the momentum of the dark component in terms of its lifting are required. The detailed formulation and the statements of the main results are provided in Section~\ref{sec:statements}. 

The rest of the paper is organized as follows. First, Section~\ref{sec:rearrangements} is devoted to the statements of well--known preliminary results on symmetric decreasing rearrangements. After that, Section~\ref{sec:properties} includes the proof of Theorem~\ref{thm:properties}. Next, Section~\ref{sec:compactness} contains the proof of the compactness of minimizing sequences. Finally, we prove Theorem~\ref{thm:mainresult} in Section~\ref{sec:proof}.

\medskip

\noindent{\bf Acknowledgments:} The author would like to thank the members of the Instituto Superior T\'ecnico (Universidade de Lisboa), particularly Sim\~ao Correia and Hugo Tavares, for their support and hospitality during the research stay in which this work was carried out.

\medskip
{
\noindent{\bf Funding:} This work has been supported by the Grant PID2024--156079NA--I00 of the MCIU/AEI (Spain).
}

\medskip

\noindent{\bf Notation:} Unless a complex range is indicated explicitly, the Lebesgue and Sobolev  spaces appearing in this work are of real--valued functions. If $w\in L^r(\R)$ for some $r\in [1,\infty]$, we denote $\|w\|_r=\|w\|_{L^r(\R)}$. Likewise, if $w\in H^1(\R)$, we write $\|w\|_{1,2}=\|w\|_{H^1(\R)}$.

\section{Functional setting and main results}\label{sec:statements}

Let us define the energy space
\[\boE(\R)=\{u\in H^1_{\loc}(\R;\C):\, u'\in L^2(\R;\C),\, 1-|u|^2\in L^2(\R;\C)\},\]
and the nonvanishing energy space
\[\NE=\{u\in\boE(\R):\, |u|>0\}.\]
The fundamental properties of these spaces are derived in \cite{gerardenergyspace}. For instance, we stress that $\boE(\R)\subset L^\infty(\R)$, and also that $\boE(\R)$ contains oscillating functions at infinity, which in particular implies that the energy space is not a vector space, not even an affine space. On the other hand, every function in the nonvanishing energy space admits a lifting. More precisely, for every $u\in\NE$, there exists $\theta\in H^1_{\loc}(\R)$ such that $\theta'\in L^2(\R)$ and $u=\rho e^{i\theta}$ in $\R$, where $\rho=|u|$. We will use this property and this notation repeatedly throughout the paper.

For $u=\rho e^{i\theta}\in\NE$, let us introduce the functional
\begin{equation}\label{eq:momentum}
p(u)=\frac12\int_\R G(|1-\rho|)\theta'\d x,\quad\text{where }\quad G(s)=s(2-s),\, s\geq 0.
\end{equation}
Notice that $G$ is increasing in $[0,1]$. Observe also that $G(|1-\rho|)$ recasts as
\begin{equation}\label{eq:recastmomentum}
    G(|1-\rho|)=1-\rho^2  + 4\chi_{\{\rho>1\}}(\rho-1),
\end{equation}
where $\chi$ denotes the indicator function. Therefore, if $\rho\leq 1$ in $\R$, it follows that $p(u)=Q(u)$, where $Q$ is the standard renormalized momentum \cite{bethuel0}
\begin{equation}\label{eq:usualmomentum}
Q(u)=\frac12\int_\R(1-\rho^2)\theta'\d x=-\frac12\int_\R\langle iu',u\rangle\frac{1-|u|^2}{|u|^2}\d x.
\end{equation}
The sort of redefined momentum \eqref{eq:momentum} is suitable for applying symmetric decreasing rearrangements, as can be verified in Section~\ref{sec:properties}. In any case, we will show that eventually the solutions we obtain have modulus smaller than one.

Let us define the sets
\[\boX=\{(u,v)\in \NE\times H^1(\R):\, |u|< 2\},\quad \boX_{q,m}=\{(u,v)\in \boX:\, \,p(u)=q,\, \|v\|^2_2=m\},\]
for $q\geq 0$, $m\geq 0$. Notice that we consider only real--valued functions for the second component of the pairs in $\boX$. We also stress that, if $u\in\NE$ and $|u|<2$, then $|1-|u||\in (0,1)$, which is precisely the interval where $G$ in \eqref{eq:momentum} is increasing. This idea will be behind some proofs of  Section~\ref{sec:properties}. 

We will look for critical points in $\boX_{q,m}$ of the renormalized energy functional
\begin{equation}\label{eq:energy}
E(u,v)=\int_{\R} \Big(\frac12|u'|^2+\frac12(v')^2+\frac14(1-|u|^2)^2+\frac{\beta}{4}v^4-\frac{\alpha}{2}(1-|u|^2)v^2\Big)\d x,
\end{equation}
which is finite for every $u\in\boE(\R)$, $v\in H^1(\R)$. Moreover, for $u=\rho e^{i\theta}\in\NE$, the energy recasts as
\[E(u,v)=\int_{\R} \Big(\frac12(\rho')^2+\frac14(1-\rho^2)^2+\frac12\rho^2(\theta')^2+\frac12(v')^2+\frac{\beta}{4}v^4-\frac{\alpha}{2}(1-\rho^2)v^2\Big)\d x.\]
Notice that, if $(u,v)\in\boX_{q,m}$, then $\int_\R|u|^2v^2\d x=m-\int_\R(1-|u|^2)v^2\d x$.
Thus, the proper energy \eqref{eq:originalenergy} and the renormalized energy \eqref{eq:energy} differ only in a constant. From now on, we will frequently just call energy to \eqref{eq:energy}.

Observe that 
\[E(u,v)\geq -\frac{\alpha}{2}m,\quad\text{for every }(u,v)\in \boX_{q,m}.\]
Therefore, the energy is bounded from below in $\boX_{q,m}$ and, as a consequence, one may define the miminizing surface
\[\Emin(q,m)=\inf\{E(u,v):\, (u,v)\in \boX_{q,m}\}.\]

Let us first analyze the semitrivial cases $\Emin(0,m)$ and $\Emin(q,0)$. On the one hand, {in the miscible regime $\alpha^2\leq \beta$, it is simple to prove (see Theorem~\ref{thm:properties})} that
\[\Emin(0,m)=0,\quad\text{for every }m\geq 0,\]
and the infimum is never attained if $m>0$. On the other hand, for $m=0$ we have the following well--known result:

\begin{theorem}[\cite{bethuel0}]\label{thm:scalar}
    For every $q\in (0,\pi/2)$ there exists $\bu_q\in\NE$, with $0<|\bu_q|<1$, $p(\bu_q)=q$, such that 
    \begin{equation}\label{eq:energygray}
        E(\bu_q,0)=\Emin(q,0)<\sqrt{2}q.
    \end{equation}
    Moreover, there exists $c=c(q)\in (0,\sqrt{2})$ such that $E(\bu_q,0)=(2-c^2)^\frac32/3$, and $\bu_q$ is the unique solution, up to invariances, to \eqref{eq:u}. These solutions are explicit and are given by 
\begin{equation}\label{eq:explicitsol}
    \bu_q(x)=\sqrt{\frac{2-c^2}{{2}}}\tanh\left(\frac{\sqrt{2-c^2}}{2}x\right)-\frac{c}{\sqrt{2}}i.
\end{equation}
    Besides, there exists $\bu_{\pi/2}\in\boE(\R)$ such that
    \begin{equation}\label{eq:energyblack}
        E(\bu_{\pi/2},0)=\frac{\sqrt{8}}{3}=\min\left\{E(u):\,u\in\boE(\R),\,\inf_\R|u(x)|=0\right\},
    \end{equation}
    which is the unique solution, up to invariances, to \eqref{eq:u} with $c=0$, and whose explicit expression is \eqref{eq:explicitsol} with $c=0$.
\end{theorem}

    \begin{remark}
        We note that the classical minimization problem establishes the renormalized momentum $Q$, given by \eqref{eq:usualmomentum}, as a constraint. Specifically, 
    \begin{equation*}
    E(\bu_q,0)=\inf\{E(u,0):\, u\in \NE,\, Q(u)=q\}:=F_{\textup{min}}(q).
    \end{equation*}
    However, we will show in Lemma~\ref{lemma:sameinf} that the infimum in $\Emin(q,0)$ may equivalently be taken over functions $u\in\NE$ such that $|u|\leq 1$. Therefore, in light of \eqref{eq:recastmomentum}, and taking into account that $0<|\bu_q|<1$ in $\R$, it follows 
    \[E(\bu_q,0)=F_{\textup{min}}(q)\leq\inf\{E(u,0):\, 0<|u|\leq 1,\, Q(u)=q\}=\Emin(q,0)\leq E(\bu_q,0).\]
    Therefore, the identity $E(\bu_q,0)=\Emin(q,0)$ in the statement of Theorem~\ref{thm:scalar} remains entirely valid.
    \end{remark}

The properties of the scalar dark soliton stated in the previous result will be used to obtain fundamental regularity and monotonicity properties of the minimizing surface $\Emin(q,m)$. These include the key subadditivity property required for ruling out dichotomy. The properties are gathered in the following result:

\begin{theorem}\label{thm:properties}
	Let $\alpha>0$ and $\beta > 0$ satisfy $\alpha^2\leq\beta$. 
	\begin{enumerate}
		\item \label{item:monotoneq} For every $m\geq 0$, the function $q\mapsto \Emin(q,m)$ is nondecreasing in $[0,\infty)$ and $\Emin(0,m)=0$. In particular, $\Emin(q,m)\geq 0$ for all $q\in [0,\pi/2)$, $m\geq 0$.
		\item \label{item:monotonem} For every $q\geq 0$, the function $m\mapsto \Emin(q,m)+\alpha m/2$ is nondecreasing in $[0,\infty)$. 
		\item \label{item:decreasingm} For $0\leq  q_1\leq q_2<\pi/2$ and $0\leq m_1\leq m_2$, one has
			\begin{equation}\label{lipsineq}
				\Emin(q_2,m_2)\leq \Emin(q_1,m_1)+\sqrt{2}(q_2-q_1).
			\end{equation}
            In particular, for every $q\in [0,\pi/2)$, the function $m\mapsto \Emin(q,m)$ is nonincreasing in $[0,\infty)$, and
            \begin{equation}\label{eq:energeticallyfav}
            \Emin(q,m)<\Emin(q,0),\quad\text{for all }q\in (0,\pi/2),\, m>0.
            \end{equation}
		\item  \label{item:lipschitz} The function $(q,m)\mapsto\Emin(q,m)$ is Lipschitz in $[0,\pi/2)\times[0,\infty)$.
        \item \label{item:subadd} For $q_1\in [0,\pi/2)$, $q_2\in [0,\pi/2)$, $m_1\geq 0$, and $m_2\geq 0$, it follows 
    \begin{equation}\label{eq:subadditive}
    \Emin(q_1+q_2,m_1+m_2)\leq\Emin(q_1,m_1)+\Emin(q_2,m_2),
    \end{equation}
    and the equality in \eqref{eq:subadditive} holds if, and only if, $(q_1+q_2)(q_1+m_1)(q_2+m_2)=0$.
    
	\end{enumerate}
\end{theorem}

The properties of $\Emin$ ultimately lead to the compactness of the minimizing sequences and, in turn, to the existence of a finite--energy solution to \eqref{TWS}--\eqref{condinftyuv}. For instance, we highlight that \eqref{eq:energeticallyfav} plays a central role in guaranteeing the nonvanishing of minimizing sequences. This inequality can be formally interpreted as expressing an \emph{energetic advantage} of nontrivial bright components over those with vanishing mass.

The main result of the paper is the following:

\begin{theorem}\label{thm:mainresult}
    Let $\alpha>0$ and $\beta > 0$ satisfy $\alpha^2\leq\beta$, and let $q\in (0,\pi/2)$ and $m>0$. Assume one of the following two conditions:
    \begin{align} \label{eq:smallnesscondq}\tag{H1}
        &\alpha^2<\beta,\quad \Emin(q,m)< \Big(1-\frac{\alpha^2}{\beta}\Big)\frac{\sqrt{8}}{3},
        \\
     \label{eq:smallnesscondm}\tag{H2}
         &\Emin(q,m)+\frac{\alpha m}{2}<\frac{\sqrt{8}}{3}.
    \end{align}
    Then, there exists $(u,v)\in \boX_{q,m}$ such that $\Emin(q,m)=E(u,v)$. Moreover, $0<|u|<1$ in $\R$, $v>0$ in $\R$, and $1-|u|$ and $v$ are even and radially nonincreasing. Furthermore, there exist $c,\lambda\in\R$ such that
    \begin{equation}\label{eq:mulambdaestimates}
    0<c<\sqrt{2},\quad  \frac{c^2}{2}\alpha<\lambda< 2\alpha+\sqrt{32}\frac{q}{m},
    \end{equation}
    and $(u,v)$ is a solution to \eqref{TWS}. 
\end{theorem}

Notice that, as expected from the scalar case, we obtain a subsonic propagation speed, namely $c\in (0,\sqrt{2})$. The proof of this fact follows from ODE arguments, as in \cite{Maris2006}.

{
Some comments about conditions \eqref{eq:smallnesscondq} and \eqref{eq:smallnesscondm} are in order. First, under condition \eqref{eq:smallnesscondq}, our main result yields the existence of a dark--bright soliton provided that $\Emin(q,m)$ is small enough. As shown by \eqref{eq:energeticallyfav}, a sufficient condition for \eqref{eq:smallnesscondq} is 
\[\alpha^2<\beta,\quad \Emin(q,0)\leq  \Big(1-\frac{\alpha^2}{\beta}\Big)\frac{\sqrt{8}}{3},\]
which means that $q$ is small, while $m$ can be as large as desired.

Our result also follows under condition \eqref{eq:smallnesscondm}. Notice that, by virtue of Theorem~\ref{thm:scalar}, this condition is trivially satisfied for $m=0$. Moreover, as shown by item \ref{item:monotoneq} in
Theorem~\ref{thm:properties}, $\Emin(q,m)\geq 0$ for all $q\in (0,\pi/2)$ and $m>0$, which implies that \eqref{eq:smallnesscondm} means a smallness requirement for the mass. It follows also from \eqref{eq:energeticallyfav} in Theorem~\ref{thm:properties} that a sufficient condition for \eqref{eq:smallnesscondm} is
\[\Emin(q,0)+\frac{\alpha m}{2}\leq\frac{\sqrt{8}}{3},\]
where, recall, $\Emin(q,0)$ is explicit, see Theorem~\ref{thm:scalar}. 

Regarding the necessity of conditions \eqref{eq:smallnesscondq} and \eqref{eq:smallnesscondm}, we remark that they are certainly not necessary for the mere existence of dark--bright solitons to \eqref{GPS} in the regime $\alpha^{2}\leq \beta$. Indeed, the Manakov case ($\alpha=\beta=1$) is not covered by our assumptions when $m$ is large, even though explicit dark--bright solitons exist for all $m$, see \cite{Busch_Anglin, Sheppard_Kivshar}. However, the Manakov system possesses special features, namely its integrability and its transitioning nature between the miscible and immiscible regimes, so that a general existence proof allowing $\alpha^{2}<\beta$ may in principle require additional hypotheses. In any case, it remains unclear whether conditions \eqref{eq:smallnesscondq} and \eqref{eq:smallnesscondm} are essential for our approach or simply technical.
}

We conclude this section with a brief discussion on the stability of dark--bright solitons. To the best of the author's knowledge, a rigorous analytical proof of their orbital stability remains an open problem, even in the Gross--Clark $(\beta=0)$ and the Manakov $(\alpha=\beta=1)$ models. A standard strategy for establishing the orbital stability of constrained minimizers is the approach developed in \cite{cazlions}. However, this method is not directly applicable here, as the quantity $p(u)$, which is fixed in our variational framework, is not conserved along the flow. {It is worth noting that the global renormalized momentum
\begin{equation}\label{eq:grmomentum}
    Q(\Psi)+P(\Phi):=-\frac12\int_\R\langle i\partial_x\Psi,\Psi\rangle\frac{1-|\Psi|^2}{|\Psi|^2}\d x + \frac12\int_\R \langle i\partial_x\Phi,\Phi\rangle\d x,
\end{equation}
is indeed a conserved quantity, at least formally. Consequently, in order to prove orbital stability, it would be desirable to establish a variational minimization principle at fixed global renormalized momentum that yields fully nontrivial minimizers. However, our approach based on symmetric decreasing rearrangements does not appear to be well suited to the functional \eqref{eq:grmomentum} as a constraint.} Nevertheless, the variational structure underlying our existence result suggests, at least on a heuristic level, that orbital stability should hold, and we hope our approach lays the groundwork for a future rigorous proof.

\section{Basics of symmetric decreasing rearrangements}\label{sec:rearrangements}

In order to be self--contained, let us state some basic definitions and properties of symmetric decreasing rearrangements. More details can be found in \cite{Burchard2009,LiebLoss2001}.

\begin{definition} 
	
    Given a measurable function $f:\R\to\R$, we say that it \emph{vanishes at infinity} if $|\{|f|>t\}|<\infty$ for all $t>0$. If $f$ vanishes at infinity, its \emph{symmetric decreasing rearrangement} is defined as
\[f^\star(x)=\int_0^\infty\chi_{\{|f|>t\}^\star}(x)\d t,\]
where $\{|f|>t\}^\star$ is the open interval centered at zero whose length coincides with $|\{|f|>t\}|$, and $\chi$ denotes the indicator function.
\end{definition}

One deduces straightaway from the definition that $f^\star$ is a nonnegative, even and lower semicontinuous function which is also nonincreasing in $(0,\infty)$. Furthermore, one of the fundamental properties of the symmetric decreasing rearrangements is that, for every $t>0$,  the following identity holds:

\begin{equation*}
	\{f^\star>t\}=\{|f|>t\}^\star,\quad\text{and therefore,}\quad|\{f^\star>t\}|=|\{|f|>t\}|.
\end{equation*}
As a consequence, one has the following relation.

\begin{proposition}\label{prop:integrals}
Let $f:\R\to\R$ be a measurable function that vanishes at infinity, and let $G:[0,\infty)\to\R$ be a function such that $G=G_1-G_2$, where $G_1$ and $G_2$ are both nondecreasing and, for at least one $j\in\{1,2\}$, one has that $\int_\R G_j(|f(x)|)\d x$ is finite. Then, 
\begin{equation}\label{sdr:integralidentity}
	\int_\R G(|f(x)|)\d x=\int_\R G(f^\star(x))\d x.
\end{equation}
In particular, $\|f\|_p=\|f^\star\|_p$ for every $1\leq p\leq \infty$.
\end{proposition}

The identity \eqref{sdr:integralidentity} can alternatively be proved by applying the following elementary property.

\begin{proposition}\label{prop:compositions}
    If $f:\R\to\R$ is a measurable function that vanishes at infinity and $G:[0,\infty)\to [0,\infty)$ is nondecreasing with $G(0)=0$, then $G\circ |f|$ is measurable, it vanishes at infinity, and $(G\circ |f|)^\star=G\circ f^\star$.
\end{proposition}

The following result is commonly known as the Hardy--Littlewood inequality.

\begin{theorem}\label{thm:HardyLittlewood}
    If $f,g:\R\to [0,\infty)$ are measurable functions that vanish at infinity, then
    \[\int_\R fg\d x\leq \int_\R f^\star g^\star\d x.\]
\end{theorem}

We also recall the well--known P\'olya--Szeg\H{o} inequality.

\begin{theorem}\label{thm:PolyaSzego}
If $f\in W^{1,p}(\R)$ for some $p\in [1,\infty)$, then $f^\star\in W^{1,p}(\R)$ and
\[\int_\R|(f^\star)'|^p \d x\leq \int_\R|f'|^p \d x.\]
\end{theorem}

We finally state the version in \cite{Albert_Bhattarai} of the key result due to \cite{Garrisi}, which essentially says that P\'olya--Szeg\H{o} inequality is strict for functions that decompose in the sum of disjoint bubbles. 

\begin{theorem}\label{thm:key}
    Let $f,g\in \boC_0^\infty(\R)$ be even, nonnegative and radially nonincreasing. Let $x_1,x_2\in\R$ be such that $f(\cdot+x_1)$ and $g(\cdot+x_2)$ have disjoint supports, and let $h=f(\cdot+x_1)+g(\cdot+x_2)$. Then,
    \[\|(h^\star)'\|^2_2\leq \|f'\|^2_2 +\|g'\|^2_2 - \frac34\min\{\|f'\|^2_2,\|g'\|^2_2\}.\]
\end{theorem}

\section{Properties of the minimizing surface}\label{sec:properties}

The following key lemma assures that functions conveniently symmetrized have smaller energy.

\begin{lemma}\label{lemma:Adecreases}
    Let $\alpha>0$, $\beta\geq 0$, $q>0$, $m\geq 0$, and $(u,v)\in \boX_{q,m}$, with $u=\rho e^{i\theta}$. Then, there exists $\gamma\in (0,1]$ such that, fixing any $x_0\in\R$ and defining
    \[\tilde u=\tilde\rho e^{i\tilde\theta},\quad\text{where}\quad 1-\tilde\rho=(1-\rho)^\star,\quad \tilde\theta=\gamma\int_{x_0}^x (\theta')^\star(y) \d y,\]
    it follows that $(\tilde u,v^\star)\in \boX_{q,m}$ and $E(\tilde u,v^\star)\leq E(u,v)$.
\end{lemma}

\begin{proof}
    We will start by proving that $\tilde u\in\NE$ and $|\tilde u|<2$ in $\R$. To do so, let us first take $x_0\in\R$ such that $|1-\rho(x_0)|=\|1-\rho\|_\infty$. From Proposition~\ref{prop:integrals}, it follows 
    \[1-\tilde\rho\leq\|1-\tilde\rho\|_\infty=\|(1-\rho)^\star\|_\infty=\|1-\rho\|_\infty=|1-\rho(x_0)|.\]
    Due to the fact that $0<\rho<2$ in $\R$, it is easy to see that $|1-\rho(x_0)|<1$ and, in turn, $\tilde\rho>0$ in $\R$. Besides, since $(1-\rho)^\star\geq 0$, it follows that $\tilde\rho\leq 1<2$ in $\R$. 
    
    Next, by definition of $G$ in \eqref{eq:momentum}, we get $1-\tilde\rho^2=G(1-\tilde\rho)=G((1-\rho)^\star).$ Hence, using that $G$ is increasing in the interval $[0,1]$, and thus so is $G^2$, Proposition~\ref{prop:integrals} yields
    \begin{align*}
        \int_\R(1-\tilde\rho^2)^2\d x&=\int_\R G((1-\rho)^\star)^2 \d x=\int_\R G(|1-\rho|)^2 \d x
        \\
        &=\int_\R\Big(4(1-\rho)^2-4|1-\rho|^3+(1-\rho)^4\Big)\d x
        \\
        &\leq \int_\R\Big(4(1-\rho)^2-4(1-\rho)^3+(1-\rho)^4\Big)\d x=\int_\R (1-\rho^2)^2\d x.
    \end{align*}
    In particular, {$1-\tilde\rho^2\in L^2(\R)$}. In addition, Theorem~\ref{thm:PolyaSzego} implies
    \[\int_\R(\tilde\rho')^2\d x=\int_\R(((1-\rho)^\star)')^2\d x\leq \int_\R(\rho')^2 \d x.\]
    That is, $\tilde\rho'\in L^2(\R)$. Since $\tilde\theta'\in L^2(\R)$ too, it follows that $\tilde u\in\NE$.

    On the other hand, by Proposition~\ref{prop:compositions} and Theorem~\ref{thm:HardyLittlewood}, we deduce
    \begin{equation*}
        \int_\R G(1-\tilde\rho)(\theta')^\star \d x=\int_\R G(|1-\rho|)^\star(\theta')^\star \d x\geq \int_\R G(|1-\rho|)|\theta'| \d x\geq 2 p(u)=2q>0.
    \end{equation*}
    Therefore, the real number
    \[\gamma=2q\Big(\int_\R  G(|1-\rho|)^\star (\theta')^\star \d x\Big)^{-1},\]
    is well--defined, $\gamma\in (0,1]$ and $p(\tilde u)=q$.
    Furthermore, Proposition~\ref{prop:integrals} assures that $v^\star\in L^2(\R)$ with $\|v^\star\|_2^2=m$. Hence, $(\tilde u,v^\star)\in\boX_{q,m}$. 
    
    Let us now estimate the remaining terms in $E(\tilde u,v^\star)$. Thus, from Proposition~\ref{prop:integrals}, it follows 
    \[\int_\R(\tilde\theta')^2\d x=\gamma^2\int_\R ((\theta')^\star)^2\d x=\gamma^2\int_\R (\theta')^2\d x.\]
    In addition, by Proposition~\ref{prop:compositions} and Theorem~\ref{thm:HardyLittlewood}, we obtain
    \begin{align*}
        \int_\R(1-\tilde\rho^2)(\tilde\theta')^2\d x&=\gamma^2\int_\R  G(|1-\rho|)^\star((\theta')^2)^\star \d x\geq \gamma^2\int_\R  G(|1-\rho|)(\theta')^2\d x
        \\
        &=\gamma^2\int_\R|1-\rho|(2-|1-\rho|)(\theta')^2\d x
        \\
        &\geq\gamma^2\int_\R(2(1-\rho))-(1-\rho)^2)(\theta')^2\d x=\gamma^2\int_\R(1-\rho^2)(\theta')^2\d x.
    \end{align*}
    Therefore, recalling that $\gamma\in(0,1]$, we derive
    \begin{equation*}
        \int_\R\tilde\rho^2(\tilde\theta')^2\d x=\int_\R(\tilde\theta')^2\d x - \int_\R(1-\tilde\rho^2)(\tilde\theta')^2\d x\leq \gamma^2\int_\R\rho^2(\theta')^2\d x\leq\int_\R\rho^2(\theta')^2\d x.
    \end{equation*}
    Arguing as above, we also derive
    \[\int_\R(1-\tilde\rho^2)(v^\star)^2\d x\geq \int_\R(1-\rho^2)v^2\d x,\quad \int_\R ((v^\star)')^2\d x\leq\int_\R (v')^2\d x,\quad \int_\R (v^\star)^4\d x=\int_\R v^4\d x.\]
    We conclude that $E(\tilde u,v^\star)\leq E(u,v)$.
\end{proof}

Let us introduce the set
\[\boX^0_{q,m}=\left\{(\rho e^{i\theta},v)\in\boX_{q,m}\left|\begin{array}{l} 1-\rho,\theta',v \text{ are even, nonnegative, radially nonincreasing}\\ \text{functions in } \boC_0^\infty(\R),\text{ with } p(\rho e^{i\theta})=q,\, \|v\|^2_2=m\end{array}\right.\right\}.\]
Next lemma shows that, even if $\boX_{q,m}^0$ is contained in $\boX_{q,m}$, the infima of the energy in both sets coincide. Thus, one is allowed to take minimizing sequences of compactly supported, symmetric functions. We stress that one cannot expect the infimum in $\boX_{q,m}^0$ to be attained since solutions in $\boX_{q,m}$ cannot have compact support (see the proof of Theorem~\ref{thm:mainresult}).

\begin{lemma}\label{lemma:sameinf}
    For every $\alpha>0$, $\beta\geq 0$, $q>0$, $m\geq 0$, the following identity holds: \[\Emin(q,m)=\inf\{E(u,v):\, (u,v)\in \boX^0_{q,m}\}.\]
\end{lemma}

\begin{proof}
    Let $\varepsilon>0$, $q>0$ and $m\geq 0$, and let $(u,v)\in \boX_{q,m}$ be such that
    \[E(u,v)<\Emin(q,m)+\varepsilon.\]
    By Lemma~\ref{lemma:Adecreases}, $u$ can be taken such that $|u|\leq 1$. Since $\boC_0^\infty(\R)$ is dense in $H^1(\R)$ and in $L^2(\R)$, there exist $1-\rho_1,\zeta_1,v_1\in \boC_0^\infty(\R)$ such that $\rho_1>0$ and, if $\theta_1$ is any primitive of $\zeta_1$ and if $u_1=\rho_1 e^{i\theta_1}$, then
    \[E(u_1,v_1)<\Emin(q,m)+o_\varepsilon(1),\quad p(u_1)=q + o_\varepsilon(1),\quad \|v_1\|^2_2=m+o_\varepsilon(1),\]
    where  $o_\varepsilon(1)$ denotes any constant such that $\lim_{\varepsilon\to 0}o_\varepsilon(1)=0$. Next, using the notation of Lemma~\ref{lemma:Adecreases}, we have $|\tilde u_1|>0$ and
    \[E(\tilde u_1,v_1^\star)<\Emin(q,m)+o_\varepsilon(1),\quad p(\tilde u_1)=q+o_\varepsilon(1),\quad \|v_1^\star\|^2_2=m+o_\varepsilon(1).\]
    Observe that $1-|\tilde u_1|$, $\tilde\theta'_1$ and $v_1^\star$ are even, nonnegative, radially nonincreasing, and have compact support. However, they need not be smooth. Let us thus define
    \[1-\rho_2=\phi_\varepsilon*(1-|\tilde u_1|),\quad \zeta_2=\phi_\varepsilon*\tilde\theta_1',\quad u_2=\rho_2 e^{i\theta_2},\quad v_2=\phi_\varepsilon*v_1^\star,\]
    where $\theta_2$ is a primitive of $\zeta_2$,  $*$ denotes the usual convolution, and $\phi_\varepsilon \in\boC^\infty_0(\R)$ is an even and radially decreasing mollifier that we choose in such a way that
    \[E(u_2,v_2)<\Emin(q,m)+o_\varepsilon(1),\quad p(u_2):=q_\varepsilon=q+o_\varepsilon(1),\quad \|v_2\|^2_2:=m_\varepsilon=m+o_\varepsilon(1).\]
    Finally, let us take
     \[\rho_3(x)=\rho_2(x),\quad \zeta_3(x)=\frac{q}{q_\varepsilon}
     \zeta_2(x),\quad u_3=\rho_3 e^{i\theta_3},\quad v_3(x)=\sqrt{\frac{m}{m_\varepsilon}} v_2(x),\]
     where $\theta_3$ is a primitive of $\zeta_3$. It is immediate to check that $(u_3,v_3)\in \boX_{q,m}^0$ and
     \begin{align*}
         E(u_3,v_3)&=\frac12\int_\R(\rho_2')^2\d x + \frac14\int_\R (1-\rho_2^2)^2\d x +\frac{q^2}{2q_\varepsilon^2}\int_\R\rho_2^2(\theta_2')^2\d x+ \frac{m}{2m_\varepsilon}\int_\R(v_2')^2\d x
         \\
         &+\frac{\beta m^2}{4m_\varepsilon^2}\int_\R v_2^4 \d x -\frac{\alpha m}{2m_\varepsilon}\int_\R(1-\rho_2^2)v_2^2\d x=E(u_2,v_2)+o_\varepsilon(1).
     \end{align*}
    Therefore, $E(u_3,v_3)<\Emin(q,m)+o_\varepsilon(1)$. This proves the result. 
\end{proof}

In what follows, we investigate the properties of the minimizing surface. The result below establishes that it is bounded from above by $\Emin(q,0)$, being such a bound strict provided that the mass is small.

\begin{lemma}\label{lemma:smallE}
    Let $\alpha>0$, $\beta\geq 0$, $q\in(0,\pi/2)$, and $m\geq 0$. Then, 
    \begin{equation}\label{eq:eminineq}
        \Emin(q,m)\leq \Emin(q,0)<\sqrt{2}q.
    \end{equation}
    Assuming in addition 
    \begin{equation}\label{eq:alphabeta}
        m\beta< 2\alpha\int_\R(1-|\mathbf{u}_q|^2)\d x, \quad m>0,
    \end{equation}
    where $\bu_q$ is defined in \eqref{eq:explicitsol}, it follows that
    \begin{equation}\label{eq:keyineq}
        \Emin(q,m)<\Emin(q,0).
    \end{equation}
\end{lemma}

\begin{proof}
     To prove the inequality $\Emin(q,m)\leq \Emin(q,0)$, let us fix $v\in H^1(\R)$ with $v> 0$ in $\R$, $\|v\|_\infty=v(0)$, and $\|v\|_2^2=m$. For every $s>0$, let us define the function 
    \[v_s(x)=\sqrt{s}v(sx),\]
    which satisfies
    \[\|v_s\|_2^2=m,\quad \|v_{s}'\|^2_2=s^2\|v'\|_2^2,\quad\|v_s\|_4^4=s\|v\|_4^4,\quad\text{for all }s>0.\]
    Thus, we deduce 
    \begin{align}
        \nonumber\Emin(q,m)&\leq E(\bu_q,v_s)=\Emin(q,0) + \int_\R\Big(\frac12(v_s')^2+\frac{\beta}{4}v_s^4 - \frac{\alpha}{2}(1-|\bu_q|^2)v_s^2\Big)\d x
        \\
        \label{proof:negative}
        &=\Emin(q,0)+s\Big(\frac{s}{2} \|v'\|_2^2 +\frac{\beta}{4}\|v\|_4^4 -\frac{\alpha}{2}\int_\R(1-|\bu_q(x)|^2) v(sx)^2 \d x\Big).
    \end{align}
    Taking limits as $s\to 0$, we prove that $\Emin(q,m)\leq\Emin(q,0)$. Thus, \eqref{eq:energygray} leads to \eqref{eq:eminineq}. 
    
   Let us now assume \eqref{eq:alphabeta}. Notice that the dominated convergence theorem yields
    \[\lim_{s\to 0^+}\int_\R(1-|\bu_q(x)|^2)v(sx)^2\d x =  v(0)^2\int_\R (1-|\bu_q|^2)\d x.\]
    Furthermore,
    \[\|v\|^4_4\leq\|v\|^2_\infty\|v\|_2^2=v(0)^2 m.\]
    Therefore, by \eqref{eq:alphabeta}, for $s>0$ small enough, we obtain
    \[\frac{\|v'\|_2^2}{2} s +\frac{\beta}{4}\|v\|_4^4 -\frac{\alpha}{2}\int_\R(1-|\bu_q(x)|^2) v(sx)^2 \d x<0.\]
    Thus, from \eqref{proof:negative} it follows  $\Emin(q,m)<\Emin(q,0)$. 
\end{proof}

Let us now prove the properties of $\Emin$ gathered in Theorem~\ref{thm:properties}, except the strict subadditivity property \ref{item:subadd}, which will be proved later on.

\begin{proof}[Proof of Theorem~\ref{thm:properties}~\ref{item:monotoneq}-\ref{item:lipschitz}]
	To prove \ref{item:monotoneq}, let $m, q_1,$ and $q_2$ be nonnegative real numbers such that $q_1<q_2$, and let $\varepsilon>0$. Take $(u,v)\in\boX_{q_2,m}$ be such that $E(u,v)<\Emin(q_2,m)+\varepsilon$. Since $u\in\NE$, there is a lifting $u=\rho e^{i\theta}$.  Let $\lambda=q_1/q_2\in [0,1)$, and $u_\lambda=\rho e^{i\lambda\theta}$. It is clear that $p(u_\lambda)=q_1$, and
\[\Emin(q_1,m)\leq E(u_\lambda,v)\leq E(u,v)\leq\Emin(q_2,m)+\varepsilon.\]
Letting $\varepsilon$ tend to zero, we deduce that the function $q\mapsto \Emin(q,m)$ is nondecreasing in $[0,\infty)$. 

{On the other hand, given $(u,v)\in\boE(\R)\times H^1(\R)$, Young inequality and $\alpha^2\leq\beta$ imply
\begin{equation*}
\frac14(1-|u|^2)^2 + \frac{\beta}{4}v^4 -\frac{\alpha}{2}(1-|u|^2) v^2 \geq \frac14\Big(1-\frac{\alpha^2}{\beta}\Big)(1-|u|^2)^2 \geq 0.
\end{equation*}
Hence, $E(u,v)\geq 0$, so that $\Emin(q,m)\geq 0$ for all $q\in [0,\pi/2)$, $m\geq 0$. 

In addition, let $m\geq 0$ and $v\in H^1(\R)$ such that $\|v\|_2^2=m$. For every $\lambda>0$, define $v_\lambda(x)=\sqrt{\lambda}v(\lambda x)$, for all $x\in\R$. It is clear that $\|v_\lambda\|_2^2=m$ for all $\lambda>0$, and
\[E(1,v_\lambda)=\int_\R\left(\frac{\lambda^2}{2}(v')^2+\frac{\beta\lambda}{4}v^4\right)\d x.\]
Therefore $E(1,v_\lambda)\to 0$ as $\lambda\to 0$. In particular, since $(1,v_\lambda)\in\boX_{0,m}$ for all $\lambda>0$, it follows that $\Emin(0,m)=0$.} This finishes the proof of \ref{item:monotoneq}. 

Let us now prove \ref{item:monotonem}. To do so,  let $q, m_1,$ and $m_2$ be nonnegative real numbers such that $m_1<m_2$, and let $\varepsilon>0$. Take $(u,v)\in\boX_{q,m_2}$ satisfying that $E(u,v)<\Emin(q,m_2)+\varepsilon$. Let $\lambda=m_1/m_2\in [0,1)$, and $v_\lambda=\sqrt{\lambda}v$. It is clear that $\|v_\lambda\|^2_2=m_1$, and
\begin{align*}
	\Emin(q,m_1)&+\frac{\alpha m_1}{2}\leq E(u,v_\lambda)+\frac{\alpha m_1}{2}
\\ 
&=E(u,0)+\int_\R\left(\frac{\lambda}{2}(v')^2+\frac{\beta\lambda^2}{4}v^4-\frac{\alpha}{2}v^2(1-\lambda|u|^2)\right)\d x +\frac{\alpha m_2}{2}
\\
&\leq E(u,v)+\frac{\alpha m_2}{2}<\Emin(q,m_2)+\frac{\alpha m_2}{2}+\varepsilon.
\end{align*}
We conclude again by letting $\varepsilon$ tend to zero.

In order to show \eqref{lipsineq}, let $\varepsilon>0$, $0\leq q_1\leq q_2<\pi/2$, $0\leq m_1\leq m_2$, and let us distinguish three different cases:

\emph{Case 1:} Let us first assume that $0<q_1<q_2$. By Lemma~\ref{lemma:sameinf}, there exists $(u_1,v_1)\in\boX_{q_1,m_1}^0$ such that $E(u_1,v_1)<\Emin(q_1,m_1)+\varepsilon$. Analogously, denoting $s=q_2-q_1\in (0,\pi/2)$ and $t=m_2-m_1\geq 0$, there exists $(u,v)\in\boX_{s,t}^0$ such that $E(u,v)<\Emin(s,t)+\varepsilon$. Moreover, Lemma~\ref{lemma:smallE}  implies that $E(u,v)<\sqrt{2}(q_2-q_1)+\varepsilon$. Let us write $u_1=\rho_1 e^{i\theta_1}$ and $u=\rho e^{i\theta}$. There is no loss of generality (by applying a translation) in assuming that $1-\rho_1$ and $1-\rho$ have disjoint supports, and the same holds for $\theta'_1$ and $\theta'$, and for $v_1$ and $v$. Hence, defining 
\[\rho_2=\rho_1+\rho-1,\quad\theta_2=\theta_1+\theta,\quad u_2=\rho_2 e^{i\theta_2},\quad v_2=v_1+ v,\]
we deduce that $(u_2,v_2)\in \boX_{q_2,m_2}$ and
\[\Emin(q_2,m_2)\leq E(u_2,v_2)=E(u_1,v_1)+E(u,v)<\Emin(q_1,m_1)+\sqrt{2}(q_2-q_1)+2\varepsilon.\]
Since $\varepsilon$ is arbitrary, \eqref{lipsineq} follows in this case. 

\emph{Case 2:} Assume now that $q_1=q_2$. Since \eqref{lipsineq} is trivially satisfied for $q_1=q_2=0$, we may assume in addition that $q_1=q_2>0$. Arguing as for the previous case, let $(u_1,v_1)\in\boX_{q_1,m_1}^0$ be such that $E(u_1,v_1)<\Emin(q_1,m_1)+\varepsilon$. In addition, let $v\in\boC_0^\infty(\R)$ be such that $\|v\|_2^2=m_2-m_1$ and $E(1,v)<\varepsilon$ (one may find such a function $v$ by rescaling a given function with mass $m_2-m_1$, as in the proof of item \ref{item:monotoneq} or the proof of Lemma~\ref{lemma:smallE}). Applying a translation if necessary, assume that $v$ and $v_1$ have disjoint supports, and define $v_2=v_1+v$. Then, $\|v_2\|_2^2=m_2$ and
\begin{align*}
\Emin(q_2,m_2)&=\Emin(q_1,m_2)\leq E(u_1,v_2)
\\
&=E(u_1,v_1)+\int_\R\left(\frac12(v')^2+\frac{\beta}{4}v^4-\frac{\alpha}{2}v^2(1-|u_1|^2)\right)\d x
\\
&\leq E(u_1,v_1)+E(1,v)<\Emin(q_1,m_1)+2\varepsilon.    
\end{align*}
Thus, \eqref{lipsineq} follows again.

\emph{Case 3:} If $q_1=0<q_2$, {Lemma~\ref{lemma:smallE}} implies
\[\Emin(q_2,m_2)<\sqrt{2}q_2=\Emin(q_1,m_1)+\sqrt{2}(q_2-q_1).\]
Hence, \eqref{lipsineq} follows one more time.

The fact that $\Emin$ is nonincreasing in $m$ is nothing but Case 2 above. Moreover, fixing $q\in(0,\pi/2)$ and $m>0$, it follows from Lemma~\ref{lemma:smallE} the existence of $m_0\in (0,m)$ such that
\[\Emin(q,m)\leq\Emin(q,m_0)<\Emin(q,0).\]
This concludes the proof of \ref{item:decreasingm}.

We finally prove \ref{item:lipschitz}. Let $q_1,q_2\in [0,\pi/2)$ and $m_1,m_2\in [0,\infty)$. Assume that $q_1<q_2$ and $m_2<m_1$, being the remaining possibilities dealt with analogously. By \ref{item:monotoneq}, \ref{item:monotonem}, and \eqref{lipsineq}, we derive
\begin{align*}
&|\Emin(q_2,m_2) -\Emin(q_1,m_1)|\leq \frac{\alpha}{2}(m_1-m_2)+|\Emin(q_2,m_1)-\Emin(q_1,m_1)|
\\
&+\Big|\Emin(q_2,m_2)+\frac{\alpha m_2}{2}-\Emin(q_2,m_1)-\frac{\alpha m_1}{2}\Big|
\\
&= \Emin(q_2,m_1)-\Emin(q_1,m_1)+\Emin(q_2,m_1)-\Emin(q_2,m_2)+\alpha(m_1-m_2)
\\
&\leq\sqrt{2}(q_2-q_1)+\alpha(m_1-m_2).
\end{align*}
This shows that $\Emin$ is a Lipschitz function. 
\end{proof}

The following result provides the boundedness of sequences whose energy converges to $\Emin(q,m)$ and whose mass converges to $m$.

\begin{lemma}\label{lemma:energybounded}
    Let $\alpha>0$, $\beta\geq 0$, $q\in(0,\pi/2)$, and $m\geq 0$, and let $\{(u_n,v_n)\}\subset \boX$ be such that    \begin{equation}\label{eq:almostminimizingseq}
        \|v_n\|_2^2\to m,\quad E(u_n,v_n)\to\Emin(q,m),\quad\text{as }n\to\infty.
    \end{equation}
    Let us consider the lifting $u_n=\rho_n e^{i\theta_n}$. Then, for every large enough $n$, it holds     \begin{equation}\label{eq:coercive}
    2\|\rho'_n\|_{2}^2 + \|1-\rho_n^2\|_2^2 + 2\|\rho_n\theta_n'\|_2^2 + 2\|v_n'\|^2_2 + \beta\|v_n\|_4^4\leq 4\sqrt{2}q+2\alpha m.
    \end{equation}
\end{lemma}

\begin{proof}
    Observe that
    \[2\|\rho'_n\|_2^2 + \|1-\rho_n^2\|_2^2 + 2\|\rho_n\theta_n'\|_2^2 + 2\|v_n'\|^2_2 +\beta\|v_n\|_4^4=4E(u_n,v_n)-2\alpha\int_\R\rho_n^2v_n^2\d x+2\alpha \|v_n\|_2^2.\]
    From \eqref{eq:eminineq}, it clearly follows \eqref{eq:coercive}.
\end{proof}

Next result establishes a suitable positive bound from below for minimizing sequences.
This is a key step for proving a subadditivity property for $\Emin$. The second part of the lemma will also be used to rule out vanishing in the concentration--compactness argument, see Remark~\ref{remark:vanishing}.

\begin{lemma}\label{lemma:delta}
    Let $\alpha>0$, $\beta\geq 0$, $q\in(0,\pi/2)$, and $m\geq 0$, and let $\{(u_n,v_n)\}\subset \boX^0_{q,m}$ be a minimizing sequence, i.e. 
    \begin{equation}\label{eq:minimizingseq}
        E(u_n,v_n)\to\Emin(q,m),\quad\text{as }n\to\infty.
    \end{equation}
    Let us consider the lifting $u_n=\rho_n e^{i\theta_n}$. Then,  there exists $\delta>0$, independent of $n$, such that $\|\rho_n'\|_2>\delta$ for all large enough $n$. In addition, if $m>0$, then $\|v_n\|_4>\delta$ for all large enough $n$.
\end{lemma}

\begin{proof}
    We start by showing that $\|\rho_n'\|_2\geq\delta$. Arguing by contradiction, assume that, up to a subsequence, $\|\rho_n'\|_2\to 0$ as $n\to \infty$. From Young inequality, we obtain
    \[\int_\R\left(\frac14(1-\rho_n^2)^2+\frac12(\theta_n')^2\right)\d x\geq \frac{\sqrt{2}}{2}\int_\R(1-\rho_n^2)\theta_n'\d x.\]
    Since $(u_n,v_n)\in\boX_{q,m}^0$, it follows that $\rho_n\leq 1$ in $\R$, which implies that 
    \[\frac{1}{2}\int_\R(1-\rho_n^2)\theta_n'\d x=p(u_n)=q.\]
    Therefore,  \begin{align}\label{proof:Esqrt2q}
        \nonumber E(u_n,v_n)&\geq \int_\R\left(\frac12\rho_n^2(\theta_n')^2+\frac14(1-\rho_n^2)^2-\frac{\alpha}{2}(1-\rho_n^2)v_n^2\right)\d x
        \\
        &=\int_\R\left(\frac12 (\theta_n')^2+\frac14(1-\rho_n^2)^2-\frac{1}{2}(1-\rho_n^2)(\alpha v_n^2+(\theta_n')^2)\right)\d x
        \\
        \nonumber &\geq \sqrt{2}q - \frac{1}{2}\int_\R (1-\rho_n^2)(\alpha v_n^2+(\theta_n')^2)\d x.
    \end{align}
    On the other hand, by H\"older inequality, and the fact that $\{E(u_n,v_n)\}$ is bounded, we deduce
    \begin{align*}
        (1-\rho_n(x))^2&=-2\int_{-\infty}^x(1-\rho_n(y))\rho_n'(y)\d y=-2\int_{-\infty}^x\frac{(1-\rho_n(y)^2)}{1+\rho_n(y)}\rho_n'(y)\d y
        \\
        &\leq 2\|1-\rho_n^2\|_2\|\rho_n'\|_2\leq 2(4E(u_n,v_n)+2\alpha m)^\frac12\|\rho_n'\|_2\to 0,\quad\text{as }n\to\infty. 
    \end{align*}
    In particular, $\|1-\rho_n\|_\infty\to 0$ as $n\to\infty$. Moreover, since $1-\rho_n^2=(1-\rho_n)(2-(1-\rho_n))$, it also holds that $\|1-\rho_n^2\|_\infty\to 0$ as $n\to\infty$. Furthermore, since $\rho_n\geq 1-|1-\rho_n|$, it follows that $\rho_n\geq 1/2$ in $\R$, for every $n$ large enough. Hence, 
    \begin{align*}
        \int_\R(1-\rho_n^2)&(\alpha v_n^2+(\theta_n')^2)\d x \leq \|1-\rho_n^2\|_\infty\Big(m\alpha+\int_\R(\theta_n')^2\d x\Big)
        \\
        &\leq \|1-\rho_n^2\|_\infty\Big(m\alpha+4\int_\R\rho_n^2(\theta_n')^2\d x\Big)
        \\
        &\leq \|1-\rho_n^2\|_\infty(m\alpha+4(2E(u_n,v_n)+\alpha m))\to 0,\quad\text{as }n\to\infty.
    \end{align*}
    Passing to the limit in \eqref{proof:Esqrt2q}, we derive $\Emin(q,m)\geq\sqrt{2}q$, which contradicts Lemma~\ref{lemma:smallE}.

    Let us now assume that $m>0$ and show that $\|v_n\|_4\geq\delta$. Indeed, assume by contradiction that, up to a subsequence, $\|v_n\|_4\to 0$ as $n\to\infty$. Consequently, by H\"older inequality and \eqref{eq:coercive}, we derive
    \begin{equation*}
        \int_\R(1-\rho_n^2)v_n^2dx\leq \|1-\rho_n^2\|_2\|v_n\|_4^2\to 0,\quad\text{as }n\to\infty.
    \end{equation*}
    On the other hand, it is immediate to check that
    \[E(u_n,v_n)\geq E(u_n,0)-\frac{\alpha}{2}\int_\R(1-\rho_n^2)v_n^2{\d x}\geq \Emin(q,0)-\frac{\alpha}{2}\int_\R(1-\rho_n^2)v_n^2{\d x}.\]
    Passing to the limit, we conclude $\Emin(q,m)\geq \Emin(q,0)$. This contradicts \eqref{eq:energeticallyfav} in Theorem~\ref{thm:properties}, which has already been proved.
\end{proof}

We are now ready to prove the strict subadditivity property in Theorem~\ref{thm:properties}.

\begin{proof}[Proof of Theorem~\ref{thm:properties}~\ref{item:subadd}]
    If $q_1q_2=0$, then \eqref{eq:subadditive} is a direct consequence of item \ref{item:decreasingm} in Theorem~\ref{thm:properties}. Thus, let us consider the case
    $q_1q_2>0$. For $j=1,2$, let $\{(u_n^{(j)},v_n^{(j)}\}\subset \boX_{q_j,m_j}^0$ be such that 
    \[E(u_n^{(j)},v_n^{(j)})\to \Emin(q_j,m_j),\quad\text{as }n\to\infty.\]
    Let us consider the liftings $u_n^{(j)}=\rho_n^{(j)}e^{i\theta_n^{(j)}}$. Since $1-\rho_n^{(1)}$ and $1-\rho_n^{(2)}$ have compact supports, we may assume without loss of generality (by simply applying a translation) that their supports are disjoint. The same argument applies for $(\theta_n^{(1)})'$ and $(\theta_n^{(2)})'$, and for $v_n^{(1)}$ and $v_n^{(2)}$ unless one of them is zero (i.e. $m_1m_2=0$), which is a simpler case whose proof will be skipped. For some $\gamma_n>0$ to be chosen later, let us define
    \[\rho_n=1-(1-\rho_n^{(1)}+1-\rho_n^{(2)})^\star,\quad \zeta_n=\gamma_n((\theta_n^{(1)})'+(\theta_n^{(2)})')^\star,\quad v_n=(v_n^{(1)}+v_n^{(2)})^\star.\]
    Recall that $0\leq 1-\rho_n^{(j)}<1$ for $j=1,2$. Since they have disjoints supports, it follows that $0\leq 1-\rho_n^{(1)}+1-\rho_n^{(2)}<1$ too.   Proposition~\ref{prop:integrals} implies that $\|1-\rho_n^{(1)}+1-\rho_n^{(2)}\|_\infty=\|(1-\rho_n^{(1)}+1-\rho_n^{(2)})^\star\|_\infty$, so  $\rho_n>0$. Taking now $\theta_n$ as any primitive of $\zeta_n$, we define $u_n=\rho_ne^{i\theta_n}$. Taking into account that the supports are disjoint, from Proposition~\ref{prop:compositions} and Theorem~\ref{thm:HardyLittlewood} we deduce
    \begin{align*}
       \int_\R G(1-\rho_n)((\theta_n^{(1)})'+(\theta_n^{(2)})')^\star \d x&\geq\int_\R G(1-\rho_n^{(1)}+1-\rho_n^{(2)})((\theta_n^{(1)})'+(\theta_n^{(2)})') \d x
        \\
        \nonumber &=2(p(u_n^{(1)})+p(u_n^{(2)}))=2(q_1+q_2)>0.
    \end{align*}    
    In particular, for every $n$, the real number 
    \[\gamma_n=2(q_1+q_2)\Big(\int_\R G(1-\rho_n)((\theta_n^{(1)})'+(\theta_n^{(2)})')^\star \d x\Big)^{-1},\]
    is well defined, $\gamma_n\in (0,1]$, and $p(u_n)=q_1+q_2$. Moreover, Proposition~\ref{prop:integrals} implies that $\|v_n\|^2_2=m_1+m_2$.  Thus, $(u_n,v_n)\in\boX_{q_1+q_2,m_1+m_2}$ and, in turn, 
    \[\Emin(q_1+q_2,m_1+m_2)\leq E(u_n,v_n).\]
    Next we estimate $E(u_n,v_n)$ term by term. First, Theorem~\ref{thm:key} implies
    \begin{align*}
        \|\rho_n'\|_2^2&\leq \|(\rho_n^{(1)})'\|_2^2 + \|(\rho_n^{(2)})'\|^2_2 - \frac34\min\{\|(\rho_n^{(1)})'\|_2^2,\|(\rho_n^{(2)})'\|_2^2\},
        \\
        \|v_n'\|_2^2&\leq \|(v_n^{(1)})'\|_2^2 + \|(v_n^{(2)})'\|^2_2 - \frac34\min\{\|(v_n^{(1)})'\|_2^2,\|(v_n^{(2)})'\|_2^2\}.
    \end{align*}
    Furthermore,  Proposition~\ref{prop:integrals}, and the compactness of the supports, lead to
    \begin{equation*}
        \|1-\rho_n^2\|_2^2= \|1-(\rho_n^{(1)})^2\|_2^2+\|1-(\rho_n^{(2)})^2\|_2^2,\quad
        \|v_n\|_4^4= \|v_n^{(1)}\|_4^4+\|v_n^{(2)}\|_4^4.
    \end{equation*}
    In addition,  Theorem~\ref{thm:HardyLittlewood}, and again the compactness of the supports, yield
    \[\int_\R(1-\rho_n^2)v_n^2\d x\geq \int_\R(1-(\rho_n^{(1)})^2)(v_n^{(1)})^2\d x + \int_\R(1-(\rho_n^{(2)})^2)(v_n^{(2)})^2\d x.\]
    Finally, let us focus on the term 
    \[\int_\R\rho_n^2(\theta_n')^2\d x=\int_\R(\theta_n')^2\d x-\int_\R(1-\rho_n^2)(\theta_n')^2\d x.\] 
    On the one hand,
    \[\|\theta_n'\|_2^2=\gamma_n^2(\|(\theta_n^{(1)})'\|_2^2+\|(\theta_n^{(2)})'\|_2^2).\]
    On the other hand,
    \[\int_\R(1-\rho_n^2)(\theta_n')^2\d xn\geq \gamma_n^2\int_\R(1-(\rho_n^{(1)})^2)((\theta_n^{(1)})')^2\d x + \gamma_n^2\int_\R(1-(\rho_n^{(2)})^2)((\theta_n^{(2)})')^2\d x.\]
    In sum, recalling that $\gamma_n\leq 1$, we deduce
    \begin{align*}
        \int_\R\rho_n^2(\theta_n')^2\d x&\leq \gamma_n^2\int_\R(\rho_n^{(1)})^2((\theta_n^{(1)})')^2\d x + \gamma_n^2\int_\R(\rho_n^{(2)})^2((\theta_n^{(2)})')^2\d x
        \\
        &\leq \int_\R(\rho_n^{(1)})^2((\theta_n^{(1)})')^2\d x + \int_\R(\rho_n^{(2)})^2((\theta_n^{(2)})')^2\d x.
    \end{align*}
    Gathering the estimates together, we conclude
    \[E(u_n,v_n)\leq E(u_n^{(1)},v_n^{(1)})+E(u_n^{(2)},v_n^{(2)})-k_n,\]
    where 
    \begin{equation}\label{kn}
    k_n={\frac38}\Big(\min\{\|(\rho_n^{(1)})'\|_2,\|(\rho_n^{(2)})'\|_2\}+\min\{\|(v_n^{(1)})'\|_2,\|(v_n^{(2)})'\|_2\}\Big).
    \end{equation}
    This is enough to prove \eqref{eq:subadditive}. 
    
   It is obvious that \eqref{eq:subadditive} holds with equality if  $(q_1+q_2)(q_1+m_1)(q_2+m_2)=0$. Let us prove that the strict inequality holds otherwise. Indeed, let us first assume that $q_1q_2>0$.   Invoking Lemma~\ref{lemma:delta}, there exists $\delta>0$ such that $k_n>\delta$ for all $n$. Therefore,
    \[\Emin(q_1+q_2,m_1+m_2)\leq E(u_n^{(1)},v_n^{(1)})+E(u_n^{(2)},v_n^{(2)})-\delta.\]
    Thus, \eqref{eq:subadditive} holds with strict inequality if $q_1q_2>0$. 
    
    It remains to consider the case $q_1>0$ and $q_2=0$, being the case $q_2>0$ and $q_1=0$ analogous. Recall that, necessarily, $m_2>0$. Thus, bearing in mind that $\Emin(0,m_2)=0$, we are reduced to proving 
    \begin{equation}\label{reducedeq}
        \Emin(q_1,m_1+m_2)<\Emin(q_1,m_1).
    \end{equation}
    On the one hand, if $m_1=0$, then \eqref{reducedeq} holds by item \ref{item:decreasingm} in Theorem~\ref{thm:properties}, proved above. On the other hand, if $m_1>0$, Lemma~\ref{lemma:delta} implies that $\min\{\|v_n^{(1)}\|_4,\|v_n^{(2)}\|_4\}>\delta$ for some $\delta>0$. Furthermore, elementary inequalities in Lebesgue spaces yield 
    \[\delta^4<\|v_n^{(j)}\|_4^4\leq m_j\|v_n^{(j)}\|_\infty^2\leq 2m_j^2\|(v_n^{(j)})'\|_2^2,\quad j=1,2.\]
    Therefore, recalling \eqref{kn}, we deduce that $k_n>3\delta^2/({8}\sqrt{2}(m_1+m_2))$ for all $n$. This concludes the proof.
\end{proof}

\section{Compactness of minimizing sequences}\label{sec:compactness}

Next lemma provides some weak and local convergences of sequences satisfying \eqref{eq:almostminimizingseq}. It is a preliminary step towards the strong convergence of $v_n$ and $1-|u_n|^2$ in $L^2(\R)$.

\begin{lemma}\label{lemma:compactness}
Let $\alpha>0$, $\beta\geq 0$, $q\in (0,\pi/2)$, and $m\geq 0$. Assume that either \eqref{eq:smallnesscondq} or \eqref{eq:smallnesscondm} holds. Let $\{(u_n,v_n)\}\subset\boX$ satisfy \eqref{eq:almostminimizingseq}. Then, there exists $(u,v)\in\boE(\R)\times H^1(\R)$ such that, up to a subsequence,
\begin{align} \label{uconv} 
&u_n\to u, \text{ in }L^\infty_{\loc}(\R),\quad
        u_n'\rightharpoonup u',\text{ in }L^2(\R),\quad 1-|u_n|^2\rightharpoonup 1-|u|^2, \text{ in }L^2(\R),
\\
&\|u'\|_2\leq\liminf_{n\to\infty}\|u_n'\|_2,\quad \|1-|u|^2\|_2\leq\liminf_{n\to\infty}\|1-|u_n|^2\|_2, \label{uconvineq}
\\
&v_n\to v,\text{ in }L^\infty_{\loc}(\R),\quad v_n\rightharpoonup v,\text{ in }H^1(\R),\quad v_n\rightharpoonup v,\text{ in }L^r(\R),\, r\in\{2,4\}, \label{vconv}
\\
&\|v\|_{1,2}\leq\liminf_{n\to\infty}\|v_n\|_{1,2},\quad \|v\|_4\leq\liminf_{n\to\infty}\|v_n\|_4,\quad \|v\|^2_{2}\leq m, \label{vconvineq}
\\
&\int_\R |u|^2 v^2\d x\leq\liminf_{n\to\infty} \int_\R |u_n|^2 v_n^2 \d x.\label{uvconvineq}
    \end{align}
Also, there exists $\nu\in (0,1)$ such that 
\begin{equation}\label{eq:nu}
    |u_n|\geq\nu,\quad\text{in }\R,\text{ for all }n,
\end{equation}
so that $u\in\NE$ and there exists $\theta\in H^1_{\loc}(\R)$ such that $u=|u| e^{i\theta}$, $\theta'\in L^2(\R)$. Moreover, up to a subsequence,
\begin{equation}\label{thetaconv}
\theta_n'\rightharpoonup \theta',\text{ in }L^2(\R),\quad \|\theta'\|_2\leq\liminf_{n\to\infty}\|\theta_n'\|_2.
\end{equation}
\end{lemma}

{

\begin{remark}
    We underline that hypotheses \eqref{eq:smallnesscondq} and \eqref{eq:smallnesscondm} are only required to prove \eqref{eq:nu}, and then \eqref{thetaconv}, in Lemma~\ref{lemma:compactness}. In some previous works on the Gross--Pitaevskii equation with nonzero conditions at infinity, the lack of lower bounds of type \eqref{eq:nu} has been overcome by introducing different notions of the momentum which take values in $\R\setminus\pi\Z$, and then fixing this new momentum in the minimization argument instead of our renormalized momentum \eqref{eq:momentum}, see \cite{orbitalblack,deLaireGravejatSmets2024,deLaireGravejatSmets2025,MarisMur}. Although allowing the dark component to vanish could potentially yield a more general result without the assumptions \eqref{eq:smallnesscondq} and \eqref{eq:smallnesscondm}, it is not clear how to adapt an approach such as ours, based on symmetric decreasing rearrangements, if liftings of the dark components are not allowed and, in turn, the energy cannot be written in terms of $\rho$ and $\theta'$.
\end{remark}
}

\begin{proof}
Let us write $u_n=\rho_n e^{i\theta_n}$. Since $p(u_n)$ and $E(u_n,v_n)$ depend on $\theta_n$ only through its derivative, there is no loss of generality in assuming that $\theta_n(0)=0$ for all $n$. Recalling that $\|u_n'\|_2^2=\|\rho_n'\|_2^2+\|\rho_n\theta_n'\|_2^2$, it follows from \eqref{eq:coercive} in Lemma~\ref{lemma:delta} that there exists a constant $C>0$, independent of $n$, such that
    \begin{equation}\label{proof:generalest}
	\|u_n'\|_2^2+\|1-\rho_n^2\|_2^2+\|v_n'\|_2^2+\|v_n\|_4^4\leq C.
	\end{equation}
Then, a standard procedure (see \cite{bethuel0,delaire-mennuni}) leads to the existence of $u\in \boE(\R)$ such that, up to a subsequence, \eqref{uconv} and \eqref{uconvineq} hold. Analogously, there exists $v\in H^1(\R)$ such that, up to a subsequence, \eqref{vconv} and \eqref{vconvineq} hold too. Furthermore, \eqref{uvconvineq} follows from \eqref{uconv}, \eqref{vconv} and Fatou lemma.

Let us prove \eqref{eq:nu}. To do so, assume by contradiction that there exists a sequence of real numbers $\{x_n\}$ such that, up to a subsequence, $|u_n(x_n)|\to 0$ as $n\to\infty$. Consider the translated functions $\tilde u_n=u_n(\cdot+x_n)$ and $\tilde v_n=v_n(\cdot+x_n)$. Clearly, $\{(\tilde u_n,\tilde v_n)\}\subset\boX$ satisfies \eqref{eq:almostminimizingseq} too. Then, as above, there exists $\tilde u\in\boE(\R)$ such that $\tilde u_n'\rightharpoonup\tilde u'$ in $L^2(\R)$ and $\tilde u_n\to\tilde u$ in $L^\infty_{\loc}(\R)$. In particular, $|\tilde u(0)|=0$. 

Assume in the first place \eqref{eq:smallnesscondq}. Then, Young inequality implies
\begin{equation}\label{eq:young}
\frac14(1-|\tilde u_n|^2)^2 + \frac{\beta}{4}\tilde v_n^4 -\frac{\alpha}{2}(1-|\tilde u_n|^2)\tilde v_n^2 \geq \frac14\Big(1-\frac{\alpha^2}{\beta}\Big)(1-|\tilde u_n|^2)^2 \geq 0,
\end{equation}
so that 
\[\Big(1-\frac{\alpha^2}{\beta}\Big)E(\tilde u_n,0)\leq E(\tilde u_n,\tilde v_n).\]
Thus,  Fatou lemma, the weak lower semicontinuity of the $L^2$ norm, \eqref{eq:almostminimizingseq}, and \eqref{eq:smallnesscondq}, yield
\[\Big(1-\frac{\alpha^2}{\beta}\Big)E(\tilde u,0)\leq \Emin(q,m)< \Big(1-\frac{\alpha^2}{\beta}\Big)\frac{\sqrt{8}}{3}.\]
Thus, \eqref{eq:energyblack} implies that $|\tilde u|>0$ in $\R$, a contradiction with $|\tilde u(0)|=0$.

On the other hand, assume that \eqref{eq:smallnesscondm} holds. It is then obvious that
\[E(\tilde u_n,0)\leq E(\tilde u_n,\tilde v_n)+\frac{\alpha}{2}m_n,\]
where $m_n=\|v_n\|_2^2$. Passing to the limit using \eqref{eq:almostminimizingseq}, we derive
\[E(\tilde u,0)\leq \Emin(q,m)+\frac{\alpha}{2}m<\frac{\sqrt{8}}{3},\]
again a contradiction with the fact that $|\tilde u(0)|=0$.

As a consequence of \eqref{eq:nu} and \eqref{eq:coercive}, we deduce that $\{\theta_n'\}$ is bounded in $L^2(\R)$, so that there exists $\zeta\in L^2(\R)$ such that, up to a subsequence, 
\begin{equation}\label{zetaconv}
\theta_n'\rightharpoonup \zeta,\quad\text{in }L^2(\R),\quad \|\zeta\|_2\leq\liminf_{n\to\infty}\|\theta_n'\|_2.
\end{equation}
Let us show that $\zeta=\theta'$. Indeed, applying \eqref{zetaconv}, and recalling that $\theta_n(0)=0$, we deduce
\[\theta_n(x)=\int_0^x \theta_n'(y)\d y=\int_\R\theta_n'(y)\chi_{[x^-,x^+]}(y)\d y\to \int_0^x\zeta(y)\d y,\]
for a.e. $x\in\R$, where $x^-=\min\{x,0\}$ and $x^+=\max\{x,0\}$, and $\chi$ is the indicator function. Since we also have that $|u_n(x)|\to|u(x)|\geq\nu$ and $u_n(x)\to u(x)=|u(x)| e^{i\theta(x)}$ for a.e. $x\in\R$, we conclude that $e^{i\theta(x)}=e^{i\int_0^x\zeta(y)\d y}$. That is, $\theta(x)=\int_0^x\zeta(y)\d y+2k\pi$ for some $k\in\mathbb{Z}$, and thus, $\theta'=\zeta$. This proves \eqref{thetaconv}.
\end{proof}

Next theorem leads to the compactness of {sequences satisfying \eqref{eq:almostminimizingseq}} and, in turn, to the existence of a constrained energy minimizer.

\begin{theorem}\label{thm:compactness}
    Let $\alpha>0$ and $\beta > 0$ satisfy $\alpha^2\leq\beta$ and let $q\in (0,\pi/2)$ and $m>0$. Assume that either \eqref{eq:smallnesscondq} or \eqref{eq:smallnesscondm}  holds. Let $\{(u_n,v_n)\}\subset\boX$ satisfy \eqref{eq:almostminimizingseq}, as well as
    \begin{equation}\label{eq:momentumconv}
        p(u_n)\to q,\quad\text{as }n\to\infty.
    \end{equation}
    Then, there exist $(u,v)\in\boX_{q,m}$ and $\{x_n\}\subset\R$ such that, up to a subsequence, 
\begin{equation}\label{eq:strongconv}
1-|u_n(\cdot + x_n)|^2\to 1-|u|^2,\text{ in }L^2(\R),\quad v_n(\cdot+x_n)\to v,\text{ in }L^2(\R),
\end{equation}
and $E(u,v)=\Emin(q,m)$. 
\end{theorem}

\begin{remark}\label{remark:vanishing}
    We emphasize that, in the proof of Theorem~\ref{thm:compactness}, the fact that $m>0$ allows to apply Lemma~\ref{lemma:delta}, which is the key to rule out vanishing. Naturally, the result is also true for $m=0$ (with $v\equiv 0$), even though the {proof of the compactness} in this case differs from ours, see \cite{bethuel0}.
\end{remark}

\begin{remark}\label{remark:stability}
    We also note that the compactness of {sequences satisfying \eqref{eq:almostminimizingseq}--\eqref{eq:momentumconv}} is often a powerful tool in establishing orbital stability, following the approach of \cite{cazlions}. Although this method is not directly applicable in our setting, since the fixed quantity $p(u)$ is not conserved along the flow in general, we have chosen to state and prove the result in the broader context of {sequences satisfying \eqref{eq:almostminimizingseq}--\eqref{eq:momentumconv}}, rather than restricting ourselves to truly minimizing sequences, in order to emphasize that this step does not present an obstacle to proving stability. 
\end{remark}

{
\begin{remark}
    We point out that, in the proof of Theorem~\ref{thm:compactness}, the condition $\alpha^2\leq \beta$ is used only to have the strict inequality in \eqref{eq:subadditive}.
\end{remark}
}

\begin{proof}[Proof of Theorem~\ref{thm:compactness}]
By Lemma~\ref{lemma:compactness}, there exists $(u,v)\in\boX$ satisfying \eqref{uconv}--\eqref{uvconvineq}. Therefore,
    \begin{equation}\label{proof:fatou}
        E(u,v)+\frac{\alpha}{2}\|v\|_2^2\leq \liminf_{n\to\infty} \Big(E(u_n,v_n)+\frac{\alpha}{2}\|v_n\|_2^2\Big)=\Emin(q,m)+\frac{\alpha m}{2}.
    \end{equation}
Thus, $E(u,v)= \Emin(q,m)$ provided that $(u,v)\in\boX_{q,m}$. We thus claim that $(u,v)\in\boX_{q,m}$, i.e. $p(u)=q$ and $\|v\|_2^2=m$. The proof follows by a concentration--compactness argument as in \cite{alhelou}{, which in turn is based on the classical method in \cite{CCPLionsI} (see also \cite{cazenave})}. Indeed, for every $n$, let us denote $u_n=\rho_n e^{i\theta_n}$ and consider the function
    \[f_n=\frac12(\rho_n')^2+\frac12\rho_n^2(\theta_n')^2+\frac14(1-\rho_n^2)^2+\frac12(v_n')^2+\frac{\beta}{4}v_n^4+\frac{\alpha}{2}(\rho_n^2+1) v_n^2.\]
    Clearly, $f_n\in L^1(\R)$ with $\|f_n\|_1=E(u_n,v_n)+\alpha m$ and, by \eqref{eq:coercive}, 
    $\{\|f_n\|_1\}$ is bounded. Passing to a not relabeled subsequence, there exists $\mu\geq 0$ such that $\|f_n\|_1\to\mu$ as $n\to\infty$. Actually, since $f_n\geq \alpha v_n^2/2$, it follows that $\|f_n\|_1\geq \alpha m/2$, so that $\mu\geq \alpha m/2>0$.
    
    Let us consider the concentration function
    \[F_n(t)=\sup_{y\in\R}\int_{y-t}^{y+t}f_n(x)\d x,\quad t\geq 0.\]
    It is standard to show that there exists a nondecreasing function $F$ and a real number $\mu_0$ such that, passing again to a not relabeled subsequence, 
    \[\lim_{n\to\infty}F_n(t)=F(t),\quad\lim_{t\to\infty}F(t)=\mu_0.\]
    Since $0\leq F_n(t)\leq\|f_n\|_1$, it clearly follows that $0\leq\mu_0\leq\mu$. Moreover, there exist sequences $\{x_n\}\subset\R$ and $\{t_n\}\subset (0,\infty)$ such that $t_n\to\infty$ and 
    \[\lim_{n\to\infty}F_n(t_n)=\lim_{n\to\infty}F_n(8t_n)=\mu_0,\quad F_n(t_n)=\int_{x_n-t_n}^{x_n+t_n}f_n(x)\d x.\]
By applying the translation $x\mapsto x+x_n$, we may assume without loss of generality that $x_n=0$ for all $n$. Let us consider a function $\varphi\in\boC_c^\infty(\R)$ such that 
    \[\varphi(x)=1,\text{ for }|x|\leq 1,\quad\varphi(x)=0,\text{ for }|x|\geq 2,\quad 0<\varphi(x)< 1,\text{ for } 1<|x|<2.\]
    For every $n$, let us define the functions
    \[1-\rho_n^{(1)}(x)=\varphi\Big(\frac{x}{t_n}\Big)(1-\rho_n(x)),\quad \zeta_n^{(1)}(x)=\varphi\Big(\frac{x}{t_n}\Big)\theta_n'(x),\quad v_n^{(1)}(x)=\varphi\Big(\frac{x}{t_n}\Big)v_n(x),\]
    and let $\theta_n^{(1)}$ be a primitive of $\zeta_n^{(1)}$. Analogously,
    \begin{align*}
        1-\rho_n^{(2)}(x)=\Big(1-\varphi\Big(\frac{x}{4t_n}\Big)\Big)(1-\rho_n(x)),\quad &\zeta_n^{(2)}(x)=\Big(1-\varphi\Big(\frac{x}{4t_n}\Big)\Big)\theta_n'(x),
        \\
        v_n^{(2)}(x)=\Big(1-\varphi\Big(\frac{x}{4t_n}\Big)\Big)v_n(x),\quad &\theta_n^{(2)}\text{ a primitive of }\zeta_n^{(2)}.
    \end{align*}
    Notice that
    \begin{equation}\label{eeq1}
    \supp(1-\rho_n^{(1)})\subset [-2t_n,2t_n],\quad\supp(1-\rho_n^{(2)})\subset\R\setminus(-4t_n,4t_n),
    \end{equation}
    and the same holds for $\zeta_n^{(j)}$ and $v_n^{(j)}$ with $j=1,2$. Also, observe that
    \begin{equation}\label{eq:max}
        \rho_n^{(j)}\leq\max\{1,\rho_n\},\quad|\zeta_n^{(j)}|\leq|\theta_n'|,\quad |v_n^{(j)}|\leq|v_n|,\quad\text{for }j=1,2.
    \end{equation}
    
    For $j=1,2$, let us  consider the functions $u_n^{(j)}=\rho_n^{(j)}e^{i\theta_n^{(j)}}$. We claim that
    \begin{equation}\label{claim:dichotomy}
    \lim_{n\to\infty}\Big|E(u_n,v_n)-E(u_n^{(1)},v_n^{(1)})-E(u_n^{(2)},v_n^{(2)})\Big|=0.
    \end{equation}
    To prove the claim, we will provide the details only for two different terms in the limit, as the rest can be dealt with analogously.

    Thus, taking \eqref{eeq1} into account, we derive    \begin{align}\label{eq:decompositionCCL}
    \int_\R &\left|\rho_n^2(\theta_n')^2  -(\rho_n^{(1)})^2((\theta_n^{(1)})')^2-(\rho_n^{(2)})^2((\theta_n^{(2)})')^2\right|\d x = \int_{\{2t_n<|x|<4t_n\}}\rho_n^2(\theta_n')^2\d x
        \\
        \nonumber & + \int_{\{t_n<|x|<2t_n\}}\left|\rho_n^2(\theta_n')^2-(\rho_n^{(1)})^2((\theta_n^{(1)})')^2\right|\d x 
        \\
        \nonumber &+  \int_{\{4t_n<|x|<8t_n\}}\left|\rho_n^2(\theta_n')^2-(\rho_n^{(2)})^2((\theta_n^{(2)})')^2\right|\d x.
    \end{align}
    In order to estimate the right--hand side of \eqref{eq:decompositionCCL}, using \eqref{eq:max}, one has 
    \[(\rho_n^{(j)})^2\leq\rho_n^2+1,\quad j=1,2,\]
    and therefore, by \eqref{eq:nu}, it follows 
    \begin{align*}
        |\rho_n^2(\theta_n')^2-(\rho_n^{(j)})^2((\theta_n^{(j)})')^2|\leq (1+2\rho_n^2)(\theta_n')^2\leq \left(\frac{1}{\nu^2}+2\right)\rho_n^2(\theta_n')^2,
    \end{align*}
    for $j=1,2$. Thus, from \eqref{eq:decompositionCCL}, we deduce
    \begin{align*}
    &\int_\R\left|\rho_n^2(\theta_n')^2-(\rho_n^{(1)})^2((\theta_n^{(1)})')^2-(\rho_n^{(2)})^2((\theta_n^{(2)})')^2\right|\d x \leq \left(\frac{1}{\nu^2}+2\right)\int_{\{t_n<|x|<8t_n\}}\rho_n^2(\theta_n')^2\d x
        \\
        &\leq  2\left(\frac{1}{\nu^2}+2\right)\int_{\{t_n<|x|<8t_n\}} f_n \d x=2\left(\frac{1}{\nu^2}+2\right)(F_n(8t_n)-F_n(t_n))\to 0,\quad\text{as }n\to\infty.
    \end{align*}
    
    On the other hand, simple computations lead to
    \begin{align*}
        (\rho_n^{(1)})'(x)&=-\frac{1}{t_n}\varphi'\left(\frac{x}{t_n}\right)(1-\rho_n(x))+\varphi\left(\frac{x}{t_n}\right)\rho_n'(x),
        \\(\rho_n^{(2)})'(x)&=\frac{1}{4t_n}\varphi'\left(\frac{x}{4t_n}\right)(1-\rho_n(x))+\left(1-\varphi\left(\frac{x}{4t_n}\right)\right)\rho_n'(x).
    \end{align*}
    Thus, omitting the variable dependence, and recalling that $|t_n|\geq 1$, we derive
    \begin{align*}
        |(\rho_n')^2-((\rho_n^{(1)})')^2|&=\left|(\rho_n')^2-\frac{1}{t_n^2}(\varphi')^2(1-\rho_n)^2-\varphi^2(\rho_n')^2+\frac{2}{t_n}\varphi'\varphi(1-\rho_n)\rho_n'\right|
    \\
    &\leq C\Big((\rho_n')^2+(1-\rho_n)^2\Big)\leq C\Big((\rho_n')^2+(1-\rho_n^2)^2\Big),
    \end{align*}
    for a large enough constant $C>0$ depending on $\|\varphi'\|_\infty$. Obviously, a similar estimate holds changing $\rho_n^{(1)}$ with $\rho_n^{(2)}$. Therefore, bearing in mind \eqref{eeq1}, we deduce
    \begin{align*}
        \int_\R&\Big|(\rho_n')^2-((\rho_n^{(1)})')^2-((\rho_n^{(2)})')^2\Big|\d x\leq C\int_{\{t_n<|x|<8t_n\}}\Big((\rho_n')^2+{(1-\rho_n^2)^2}\Big)\d x
        \\
        &\leq C\int_{\{t_n<|x|<8t_n\}} f_n \d x=C(F_n(8t_n)-F_n(t_n))\to 0,\quad\text{as }n\to\infty.
    \end{align*}
    Arguing similarly for the remaining terms, we conclude that \eqref{claim:dichotomy} holds.
    
    As a consequence,
    \begin{align}
        \nonumber E(u_n,v_n)&=E(u_n^{(1)},v_n^{(1)}) + E(u_n^{(2)},v_n^{(2)}) + a_n
        \\
        \label{ineqdichotomy}
        &\geq \Emin(q_n^{(1)},m_n^{(1)}) + \Emin(q_n^{(2)},m_n^{(2)}) + a_n,
    \end{align}
    where $\lim_{n\to\infty}a_n=0$, $q_n^{(j)}=p(u_n^{(j)})$, and $m_n=\|v_n^{(j)}\|_2^2$, for $j=1,2$. Observe that the sequences $\{q_n^{(j)}\}$ and $\{m_n^{(j)}\}$ are bounded, so they converge respectively to $q^{(j)}$ and $m^{(j)}$, up to a not relabeled subsequence. Moreover, arguing as for the proof of \eqref{claim:dichotomy}, we derive
    \[\lim_{n\to\infty}\Big|p(u_n)-p(u_n^{(1)})-p(u_n^{(2)})\Big|=0,\quad \lim_{n\to\infty}\Big|\|v_n\|_2^2-\|v_n^{(1)}\|_2^2-\|v_n^{(2)}\|_2^2\Big|=0,\quad j=1,2.\]
    It is thus clear that $q=q^{(1)}+q^{(2)}$ and $m=m^{(1)}+m^{(2)}$.
    
    Recalling now that $\Emin$ is Lipschitz by Theorem~\ref{thm:properties}, and thus continuous, we pass to the limit in \eqref{ineqdichotomy} and obtain
    \[\Emin(q,m)\geq \Emin(q^{(1)},m^{(1)}) + \Emin(q^{(2)},m^{(2)}).\]
    Invoking item \ref{item:subadd} in Theorem~\ref{thm:properties}, we deduce that  $q^{(j)}=m^{(j)}=0$ for either $j=1$ or $j=2$.

    Assume by contradiction that $q^{(1)}=m^{(1)}=0$. Then we have  $\lim_{n\to\infty}E(u_n^{(1)},v_n^{(1)})=0$ and $\lim_{n\to\infty}m_n^{(1)}=0$. Since, from \eqref{eeq1},     
    \[F_n(t_n)=\int_{-t_n}^{t_n}f_n\d x\leq E(u_n^{(1)},v_n^{(1)})+\alpha m_n^{(1)},\]
    it follows that $\lim_{n\to\infty}F_n(t_n)=0$. That is, $\mu_0=0$. Then, the fact that $F_n$ is nonnegative and nondecreasing implies that $F\equiv 0$, namely, vanishing occurs. In particular, 
    \[\lim_{n\to\infty}\sup_{y\in\R}\int_{y-t}^{y+t} ((v_n')^2+v_n
    ^2)\d x=0,\quad\text{for all }t\geq 0.\]
    As a consequence, it is an exercise to prove that
    \[\lim_{n\to\infty}\int_\R (v_n'(x+y_n)\phi'(x)+v_n(x+y_n)\phi(x))\d x= 0,\quad\text{for all }\phi\in\boC_0^\infty(\R),\text{ and all }\{y_n\}\subset\R.\]
    Then, \cite[Lemma~{1.2}]{tintarev-fieseler} implies that  
    \[v_n(\cdot+y_n)\rightharpoonup 0,\quad\text{in }H^1(\R),\quad\text{for every sequence }\{y_n\}\subset\R.\]
    Therefore, applying \cite[Lemma~{3.3}]{tintarev-fieseler}, we conclude that $v_n\to 0$ in $L^4(\R)$. This contradicts Lemma~\ref{lemma:delta}.

   Thus, necessarily $q^{(2)}=m^{(2)}=0$. In this case, by \eqref{eeq1} we have
   \[\int_{\R\setminus(-8t_n,8t_n)}f_n\d x=\|f_n\|_1 - F_n(8t_n)\leq E(u_n^{(2)},v_n^{(2)})+\alpha m_n^{(2)}.\]
   Therefore, 
   \[\lim_{n\to\infty}\int_{\R\setminus(-8t_n,8t_n)}f_n\d x=0.\]
   {Equivalently, $\mu_0=\mu$. }Then, given $\varepsilon>0$, a standard procedure {(see for instance Step 1 of the proof of \cite[Proposition~{1.7.6}]{cazenave})} provides some $t_\varepsilon>0$ such that
   \[\int_{\R\setminus(-t_{\varepsilon},t_{\varepsilon})}f_n\d x<\varepsilon,\]
   for all $n$ large enough. In particular,
   \[m-\varepsilon<\int_{-t_\varepsilon}^{t_\varepsilon}v_n^2\d x.\]
   Using the local uniform convergence in \eqref{vconv} and also \eqref{vconvineq}, we deduce
   \[m-\varepsilon\leq \int_{-t_\varepsilon}^{t_\varepsilon}v^2\d x\leq \|v\|_2^2\leq m.\]
   Since $\varepsilon$ is arbitrary, we conclude that $\|v\|_2^2=m$.

   Analogously, 
   \[\|1-\rho_n^2\|_2^2-\varepsilon<\int_{-t_\varepsilon}^{t_\varepsilon}(1-\rho_n(x)^2)^2\d x.\]
   Using the local uniform convergence in \eqref{uconv} and also \eqref{uconvineq}, we derive
   \[\limsup_{n\to\infty}\|1-\rho_n^2\|_2^2-\varepsilon\leq\int_{-t_\varepsilon}^{t_\varepsilon}(1-\rho(x)^2)^2\d x\leq \liminf_{n\to\infty}\|1-\rho_n^2\|_2^2.\]
   Since $\varepsilon$ is arbitrary, we conclude that $\lim_{n\to\infty}\|1-\rho_n^2\|_2=\|1-\rho^2\|_2$. This proves \eqref{eq:strongconv}. In particular, there exists $h\in L^2(\R)$ such that, up to a subsequence,
   \[\rho_n\to \rho,\quad |1-\rho_n^2|\leq h,\quad\text{a.e. in }\R.\]
   Consequently, the continuity of $G$ implies that 
   \[G(|1-\rho_n|)\to G(|1-\rho|),\quad\text{a.e. in }\R.\]
   Moreover, from \eqref{eq:recastmomentum}, we get
   \[|G(|1-\rho_n|)|\leq |1-\rho_n^2|+4|1-\rho_n|=\left(1+\frac{4}{1+\rho_n}\right)|1-\rho_n^2|\leq 5|1-\rho_n^2|\leq 5h.\]
   Thus, the dominated convergence theorem implies that 
   \[G(|1-\rho_n|)\to G(|1-\rho|),\quad\text{in }L^2(\R),\]
   and then, from \eqref{thetaconv}, we get $p(u_n)\to p(u)$. Hence, $p(u)=q$. As a result, $(u,v)\in \boX_{q,m}$, as we claimed.
\end{proof}

\section{Proof of the main result}\label{sec:proof}

\begin{proof}[Proof of Theorem~\ref{thm:mainresult}] Let $\{(u_n,v_n)\}\subset\boX_{q,m}^0$ be a minimizing sequence, i.e. \eqref{eq:minimizingseq} holds. Theorem~\ref{thm:compactness} yields the existence of $(u,v)\in\boX_{q,m}$ such that  $E(u,v)=\Emin(q,m)$. Furthermore, the symmetry and monotonicity of $1-|u|$ and $v$ are inherited from $1-|u_n|$ and $v_n$ via pointwise convergence. 

Let us show that $(u,v)$ solves \eqref{TWS} for some $c,\lambda\in\R$. Indeed, let us denote $w=u(1-|u|^2)/|u|^2$. It is easy to see that $w\in H^1(\R)$ and
    \[(p'(u)|w)=\int_\R\langle iu',u\rangle\frac{1-|u|^2}{|u|^2}\d x=-2p(u)=-2q\not=0.\]
    Let $z\in H^1(\R)$ be such that $(p'(u)|z)=0$, i.e. $z\in\ker(p'(u))$, and consider the function
    \[f(s,t)=p(u+sw+tz),\quad s,t\in\R,\]
    so that 
    \[f(0,0)=q,\quad \frac{\partial f}{\partial s}(0,0)=(p'(u)|w)\not=0,\quad\frac{\partial f}{\partial t}(0,0)=(p'(u)|z)=0.\]
    The implicit function theorem yields the existence of $\delta>0$ and a function $g:(-\delta,\delta)\to\R$ such that 
    \[g(0)=0,\quad g'(0)=0,\quad f(g(t),t)=q,\text{  for all }t\in(-\delta,\delta).\]
    On the other hand, we have that $0<|u+g(t)w+tz|<2$ for all small enough $t\in (-\delta,\delta)$. Thus, $u+g(t)w+tz\in\boX_{q,m}$ for all small enough $t\in (-\delta,\delta)$, and thus, the function $t\mapsto E(u+g(t)w+tz,v)$ achieves a local minimum at $t=0$. Consequently, 
    \[\left(\frac{\partial E}{\partial u}(u,v)|z\right)=0.\] 
    
    Now, let $\phi\in H^1(\R)$. From the decomposition $H^1(\R)=\ker(p'(u))\bigoplus \R w$, we have that $\phi=z+sw$ for some $z\in\ker(p'(u))$ and $s\in\R$. Therefore,
    \[\Big(\frac{\partial E}{\partial u}(u,v)|\phi\Big)=s\Big(\frac{\partial E}{\partial u}(u,v)|w\Big),\quad (p'(u)|\phi)=s(p'(u)|w).\]
    Thus, we conclude that
    \[\frac{\partial E}{\partial u}(u,v)=cp'(u),\quad\text{where }c=\Big(\frac{\partial E}{\partial u}(u,v)|w\Big)(p'(u)|w)^{-1}.\]
    Arguing similarly for the second component, we also derive
    \[\frac{\partial E}{\partial v}(u,v)=\lambda v,\quad\text{where }\lambda=\Big(\frac{\partial E}{\partial v}(u,v)|v\Big)\|v\|_2^{-2}.\]
    This completes the proof of the fact that $(u,v)$ is a solution to \eqref{TWS}, with $\lambda,c\in\R$ being Lagrange multipliers.

    We finish the proof by showing \eqref{eq:mulambdaestimates}. Indeed, since $(u,v)$ is a solution to \eqref{TWS}, arguing as in \cite{Maris2006} or \cite{dLMar2022}, it follows that $(|u|,v)=(\rho,v)$ solves 
    \begin{equation}\label{eq:mulambda}
    \begin{cases}
    \displaystyle -\rho'' + \frac{c^2 (1-\rho^4)}{4\rho^3}=(1-\rho^2-\alpha v^2)\rho, &x\in\R,
    \\
    -v''=(\lambda - \alpha \rho^2 - \beta v^2)v, &x\in\R,
    \end{cases}
    \end{equation}
{and moreover, $2\theta'=c(\rho^{-2}-1),$ where $u=\rho e^{i\theta}$.} Arguing by contradiction, assume that there exists $x_0\in\R$ such that $v(x_0)=0$. Since $v\geq 0$ in $\R$ and $v\in\boC^1(\R)$, it follows that $v'(x_0)=0$. Thus, applying Cauchy--Lipschitz theorem to the second equation in \eqref{eq:mulambda} implies that $v\equiv 0$, a contradiction. Therefore, $v>0$ in $\R$. 

Analogously, assume by contradiction that there exists $x_0>0$ such that $\rho(x_0)=1$. Since $\rho\leq 1$ in $\R$ and $\rho\in\boC^2(\R)$, it follows that $\rho'(x_0)=0$ and $\rho''(x_0)\leq 0$. The first equation in \eqref{eq:mulambda} implies that $\alpha v(x_0)^2\leq 0$, so that $v(x_0)=0$. This contradicts the fact that $v>0$ in $\R$.

Let us prove now \eqref{eq:mulambdaestimates}. First, again as in \cite{dLMar2022}, we deduce that
    \[q=p(u)=\frac{c}{4}\int_\R \frac{(1-\rho^2)^2}{\rho^2}\d x.\]
    This implies that $c>0$. Next, following \cite{Maris2006}, we multiply the first equation in \eqref{eq:mulambda} by $\rho'$, the second by $v'$, add them together and integrate, obtaining 
    \begin{equation}\label{proof:hamiltonian}
    (\rho')^2+(v')^2=\Big(1-\frac{c^2}{2\rho^2}\Big)\frac{(1-\rho^2)^2}{2} + \Big(\frac{\beta}{2}v^2+\alpha\rho^2-\lambda\Big)v^2.
    \end{equation}
    We claim that
    \begin{equation}\label{proof:rhomin}
    \rho(x)^2\geq \min\Big\{\frac{\lambda}{\alpha}-\frac{\beta}{2\alpha}v(x)^2,\frac{c^2}{2}\Big\},\quad\text{for all }x\in\R.
    \end{equation}
    Indeed, assume that there exists $x_0\in\R$ that violates \eqref{proof:rhomin}. Recall that $v(x_0)>0$. Then, evaluating \eqref{proof:hamiltonian} at $x_0$ leads to 
    \[0\leq (\rho')^2+(v')^2=\Big(1-\frac{c^2}{2\rho^2}\Big)\frac{(1-\rho^2)^2}{2} + \Big(\frac{\beta}{2}v^{2}+\alpha\rho^2-\lambda\Big)v^2< 0,\]
    a contradiction. We thus conclude that \eqref{proof:rhomin} holds.

    We next claim that there exists $y\in\R$ such that
    \begin{equation}\label{proof:lambdamuestimate}
        \frac{\lambda}{\alpha} - \frac{\beta}{2\alpha}v(y)^{2}>\frac{c^2}{2}.
    \end{equation}
    Indeed, assuming otherwise, we deduce from \eqref{proof:rhomin} that $\alpha\rho(x)^2\geq \lambda-\beta v(x)^{2}/2$ for every $x\in\R$. Hence, taking $v$ as test function in the second equation of \eqref{eq:mulambda}, we deduce
    \[0=\int_\R\Big((v')^2 + (\beta v^2+\alpha\rho^2-\lambda)v^2\Big)\d x\geq \int_\R\Big ((v')^2+\frac{\beta}{2}v^4\Big)\d x\geq 0.\]
    Therefore, $v\equiv 0$, a contradiction. In sum, \eqref{proof:lambdamuestimate} holds and, in turn, we get $2\lambda>\alpha c^2$. In fact, since $v$ is even and radially decreasing, it follows from \eqref{proof:lambdamuestimate} that 
    \[\frac{\lambda}{\alpha}>\frac{c^2}{2}+\frac{\beta}{2\alpha}v(x)^{2},\quad\text{for every }x\geq|y|.\]
    Therefore, we get from \eqref{proof:rhomin} that 
    \[\rho(x)^2\geq\frac{c^2}{2},\quad\text{for every }x\geq|y|.\]
    Since $\rho<1$, we conclude that $c^2<2$.

    Finally, Lemma~\ref{lemma:smallE} leads to
    \[\lambda m = \int_\R\Big((v')^2+\beta v^4 + \alpha \rho^2 v^2\Big)\d x\leq 4 E(u,v) + 2\alpha m<4\sqrt{2}q+ 2\alpha m.\]
    This concludes the proof.
\end{proof}

\section*{Appendix}

In this last section we prove that the bright components are always real--valued modulo a constant change of phase.

\begin{proposition}\label{prop:realvalued}
    Let $f:\R\to\R$ be a continuous function, and assume that there exists $\lambda\in\R$ such that $f(x)\leq\lambda$ for all $x\in\R$. Let $v\in H^1(\R;\C)$ be a solution to
    \begin{equation}\label{eq:linear}
    -v''=f(x)v,\, x\in\R.
    \end{equation}
    Then, there exists $\theta_0\in\R$ such that $v(x)e^{i\theta_0}\in\R$ for all $x\in\R$.
\end{proposition}

\begin{remark}
    Given any solution $(u,v)\in \boE(\R)\times H^1(\R;\C)$ to the Gross--Pitaevskii system, one may apply Proposition~\ref{prop:realvalued} with $f(x)=\lambda-\alpha|u(x)|^2-\beta|v(x)|^2$ and derive that $v$ is real--valued up to a multiplication by a constant of modulus one. 
\end{remark}

\begin{proof}[Proof of Proposition~\ref{prop:realvalued}]
    Let us assume that $v$ is not identically zero, since the result is trivial otherwise. Let $(a,b)$ be a possibly unbounded interval such that 
    \[|v(x)|>0,\,\text{for all }x\in (a,b),\quad \lim_{x\to a^+}|v(x)|=\lim_{x\to b^-}|v(x)|=0.\]
    In such an interval, we may take polar coordinates $v(x)=\rho(x)e^{i\theta(x)}$, where $\rho$ and $\theta$ are $\boC^2$ in $(a,b)$. Notice that $\theta$ may not have limits at $a$ or $b$. Writing the equation \eqref{eq:linear} in terms of $\rho$ and $\theta$, and equating real and imaginary parts, we derive
    \begin{align}
        -\rho''+(\theta')^2\rho=f(x)\rho, \quad &x\in (a,b),\label{eq:real}
        \\
        2\rho'\theta'+\rho\theta''=0,\quad &x\in (a,b).\label{eq:imaginariy}
    \end{align}
    Equation \eqref{eq:imaginariy} is equivalent to $(\rho^2\theta')'=0$ in $(a,b)$, which implies that there exists $C\in\R$ such that 
    \begin{equation}\label{eq:constant}
    \rho^2\theta'=C,\quad x\in (a,b).
    \end{equation}
    We claim that $C=0$. To prove the claim, let us substitute \eqref{eq:constant} into \eqref{eq:real} to obtain
    \begin{equation}\label{eq:rhoC}
    -\rho''+\frac{C^2}{\rho^3}=f(x)\rho, \quad x\in (a,b).
    \end{equation}
    In what follows we distinguish between the cases of $(a,b)=\R$ or $(a,b)\not=\R$. 
    
    Let us first assume that $(a,b)=\R$.  Using that $f\leq\lambda$ in $\R$, and rearranging terms in \eqref{eq:rhoC}, we obtain
    \begin{equation}\label{eq:rho''}
        \rho''\geq \frac{C^2}{\rho^3}-\lambda\rho,\quad x\in \R.
    \end{equation}
    Seeking a contradiction, assume that $C^2>0$. Let $\varepsilon>0$ be such that
    \[\frac{C^2}{\varepsilon^3}-\lambda\varepsilon:=R_\varepsilon>0.\]
    Let $x_\varepsilon\in \R$ be such that $\rho(x)<\varepsilon$ for all $x\in [x_\varepsilon,\infty)$. Then, from \eqref{eq:rhoC} we deduce
    \begin{equation}\label{eq:rho''R}
    \rho''(x)\geq R_\varepsilon,\quad x\in [x_\varepsilon,\infty).
    \end{equation}
    Then, integrating twice, we conclude
    \[\rho(x)\geq \frac{
    R_\varepsilon}{2}(x-x_\varepsilon)^2+\rho'(x_\varepsilon)(x-x_\varepsilon)+\rho(x_\varepsilon),\quad x\in [x_\varepsilon,\infty).\]
    We arrive at a contradiction by letting $x$ tend to infinity. Hence, $C=0$ and, from \eqref{eq:constant}, it follows that there exists $\theta_0\in\R$ such that $\theta=-\theta_0$ in $\R$, i.e. $v(x)e^{i\theta_0}=|v(x)|\in\R$ for all $x\in\R$. Thus, the result is proved in the case $(a,b)=\R$.

    We consider now the case $(a,b)\not=\R$. Assume first that $a\in\R$. Then, $v$ satisfies 
    \begin{equation}\label{eq:cauchy}
        \begin{cases}
            -v''=f(x)v, &x\in (a,b),
            \\
            v(a)=0.
        \end{cases}
    \end{equation}
    If we denote $v=v_1+iv_2$, it is clear that $v_j$ satisfies \eqref{eq:cauchy} too for $j=1,2$. Since the second--order equation in \eqref{eq:cauchy} is linear, it is known that the vector space of real--valued solutions to the equation has dimension 2, and since $v_1(a)=v_2(a)=0$, there exists $\mu\in\R$ such that $v_1=\mu v_2$ in $(a,b)$. This implies that $|v|=\sqrt{1+\mu^2}|v_2|$ in $(a,b)$. Notice that $v_2$ cannot change sign in $(a,b)$ since, otherwise, $|v|$ would vanish at some point in $(a,b)$. Therefore, we may assume that $v_2>0$ in $(a,b)$, an thus, 
    \[\rho=|v|=\sqrt{1+\mu^2}v_2,\quad x\in (a,b).\]
    Multiplying equation \eqref{eq:rhoC} by $(1+\mu^2)^{-1/2}$, it follows
    \[-v_2''+\frac{C^2}{(1+\mu^2)^2 v_2^3}=f(x)v_2,\quad x\in (a,b).\]
    Bearing in mind that $v_2$ satisfies also $-v_2''=f(x)v_2$ in $(a,b)$, we conclude again that $C=0$. Therefore, there exists $\theta_0\in\R$ such that $v(x)e^{i\theta_0}\in\R$ for all $x\in(a,b)$. 
    
    It is clear that the function $w:=e^{i\theta_0}v$ is a solution to \eqref{eq:linear}, it is real--valued and does not vanish in $(a,b)$, and $w(a)=0$. If $w'(a)=0$, the uniqueness of solution to the Cauchy problem associated to $-w''=f(x)w$ would imply that $w\equiv 0$, a contradiction. As a consequence, $w'(a)\not=0$. This implies that $w$ is real--valued also in $(a-\varepsilon,a)$ for some $\varepsilon>0$. Then, arguing as above, it can be proved that $w$ is real--valued in the possibly unbounded interval $(a_1,a)$ such that $|w|>0$ in $(a_1,a)$ and $\lim_{x\to a_1^+}|w(x)|=0$. An argument by induction proves that $w$ is real--valued in $(-\infty,b)$. If $b=\infty$, the proof is finished, while if $b\in\R$, the proof follows the same lines as we showed for $a\in\R$.
\end{proof}

\end{document}